\documentclass{article}

\usepackage{arxiv}          % Arxiv environment
\usepackage{texcommands}    % pre-defined tex commands

\usepackage[utf8]{inputenc} % allow utf-8 input
\usepackage[T1]{fontenc}    % use 8-bit T1 fonts
\usepackage{hyperref}       % hyperlinks
\usepackage{url}            % simple URL typesetting
\usepackage{booktabs}       % professional-quality tables
\usepackage{amsfonts}  
\usepackage{mathrsfs}       % mathscr
\usepackage{mathtools}
\usepackage{amsthm}         % theorem environments
\usepackage{graphicx}
\usepackage[dvipsnames]{xcolor}         % textcolor
\usepackage{upgreek} %varphi etc
\usepackage{bbm}
\usepackage{verbatim} %multiline comment
\usepackage{soul} %strikethrough
\usepackage[linesnumbered,ruled,vlined]{algorithm2e}
\usepackage[toc,title]{appendix} % Appendices
\usepackage{derivative} % for easy differential stuff

\newtheorem{definition}{Definition}[section]
\newtheorem{theorem}[definition]{Theorem}
\newtheorem*{theorem*}{Theorem}
\newtheorem{proposition}[definition]{Proposition}
\newtheorem{lemma}[definition]{Lemma}
\newtheorem{corollary}[definition]{Corollary}
\newtheorem{assumption}[definition]{Assumption}
\newtheorem{remark}[definition]{Remark}
\newtheorem{example}[definition]{Example}

\DeclareMathOperator{\dom}{dom}

\DeclareMathOperator{\LH}{span}
\DeclareMathOperator{\tr}{tr}
\DeclareMathOperator{\im}{im}

\title{On a class of exponential changes of measure for stochastic PDEs}

\author{
  Thorben Pieper-Sethmacher \\
  Delft Institute for Applied Mathematics \\
  Delft University of Technology \\
    Mekelweg 4, 2628 CD Delft \\
    The Netherlands \\
  \texttt{t.pieper@tudelft.nl} \\
   \And
  Frank van der Meulen \\
  Department of Mathematics \\
  Vrije Universiteit Amsterdam \\
    De Boelelaan 1111, 1081 HV Amsterdam\\
    The Netherlands \\
  \texttt{f.h.van.der.meulen@vu.nl} \\
     \And
  Aad van der Vaart \\
  Delft Institute for Applied Mathematics \\
  Delft University of Technology \\
    Mekelweg 4, 2628 CD Delft \\
    The Netherlands \\
  \texttt{a.w.vandervaart@tudelft.nl} \\
}

\begin{document}
\maketitle
\begin{abstract}
Given a mild solution $X$ to a semilinear stochastic partial differential equation (SPDE), we consider an exponential change of measure based on its infinitesimal generator $L$, defined in the topology of bounded pointwise convergence. 
The changed measure $\P^h$ depends on the choice of a function $h$ in the domain of $L$. 
In our main result, we derive conditions on $h$ for which the change of measure is of Girsanov-type. 
The process $X$ under $\P^h$ is then shown to be a mild solution to another SPDE with an extra additive drift-term. 
We illustrate how different choices of $h$ impact the law of $X$ under $\P^h$ in selected applications. These include the derivation of an infinite-dimensional diffusion bridge as well as the introduction of guided processes for SPDEs, generalizing results known for finite-dimensional diffusion processes to the infinite-dimensional case.
\end{abstract}


\section{Introduction}
\label{sec: 1_intro}

Consider a semilinear \textit{stochastic partial differential equation (SPDE)} of the form
\begin{align}
\label{eq: semilin_spde}
	\begin{split}
		\begin{cases}
			\df X(t) &= \left[  A X(t) + F(t,X(t)) \right] \df t + \sqrt{Q} \df W(t), \quad t \geq s, \\
			X(s) &= x.
		\end{cases}
	\end{split}
\end{align}
The operator $A$ denotes the generator of a strongly continuous semigroup $(S_t)_{t \geq 0}$ on a Hilbert space $H$, whereas $F$ denotes a non-linear operator and $Q$ is a symmetric, positive operator on $H$. The process $W$ is a cylindrical Wiener process on $H$, defined on a stochastic basis $(\Om, \calF, (\calF_t)_{t \geq 0},\P)$. We assume that the operators $A$, $F$ and $Q$ satisfy suitable conditions such that Equation \eqref{eq: semilin_spde} admits a unique mild solution $X = (X(t,s,x))_{t \geq s}$ for any $s \geq 0$ and $x \in H.$ 
Throughout the article we fix some arbitrary $x_0 \in H$ and simply write $X(t)$ if the SPDE in \eqref{eq: semilin_spde} is assumed to be initialized at $X(0) = x_0$.

For any $m \in \N$, let $C_m(\R_+ \times H)$ be the Banach space of continuous functions $\varphi: \R_+ \times H \rightarrow \R$ such that $\|\varphi\|_m = \sup_{t,x} (1 + \|x\|^m)^{-1} | \varphi(t,x)| < \infty.$
The process $X$ is Markovian and defines a transition semigroup
\begin{align*}
    (T_t \varphi)(s,x) = \E[ \varphi(s+t, X(t+s,s,x))], \quad s,t \geq 0, ~x \in H,
\end{align*}
on $C_m(\R_+ \times H)$. It is well-known that the semigroup $(T_t)_{t \geq 0}$ is not strongly continuous with respect to the norm topology on $C_m(\R_+ \times H)$, see e.g.\  \cite{Cerrai1994Hille} and \cite{DaPrato04Kolmogorov}. However, it does possess the properties of a strongly continuous semigroup in several weaker `modes of convergence'. This has been studied in the framework of \textit{$\calK$-convergence} in \cite{Cerrai1994Hille}, \cite{Cerrai1995Weakly} and \cite{CerraiGozzi1995Strong}, the mixed topology in \cite{GoldysKocan01Diffusion} and of \textit{bp- (bounded pointwise) or $\pi$- convergence} in \cite{Priola1999Markov}. See also \cite{GozziFabbriSwiech2017Stochastic}, Appendix B for a recent survey.  
In the respective convergence of choice, one can then define an infinitesimal generator $(L, \dom_m(L))$ of the semigroup $(T_t)_{t \geq 0}$ in the usual way.  
In this article, we will work within the framework of $\pi$-convergence as introduced in \cite{Priola1999Markov}. 

Crucially, the operator $(L, \dom_m(L))$ exhibits the common properties that are characteristic for infinitesimal generators of strongly continuous semigroups. Of particular importance for us is the fact that \textit{Dynkin's formula} holds, i.e. for any $h \in (L, \dom_m(L))$ the process
\begin{align*}
    D^h(t) = h(t,X(t)) - \int_0^t L h(s,X(s)) \, \df s
\end{align*}
is a $\P$-martingale. In other words, the process $X$ solves the \textit{martingale problem} of $(L, \dom_m(L))$ as introduced in \cite{StroockVaradhan97Multidimensional}.
Furthermore, one can show that for any positive $h \in \dom_m(L)$, the process
\begin{align*}
    E^h(t) = \dfrac{h(t,X(t))}{h(0,x_0)} \exp\left( - \int_0^t \dfrac{L h}{h}(s,X(s)) \, \df s \right), \quad t \geq 0,
\end{align*}
whenever existent, is a positive, continuous local $\P$-martingale with $\E[ E^h(0)] = 1$. If $E^h$ is a true $\P$-martingale, it defines an \textit{exponential change of measure} $\P^h$ on $\calF$ such that for any $t \geq 0$
\begin{align}
\label{eq: intro_changeofmeasure}
    \df \P^h_{\mid \calF_t} = E^h(t) \df \P_{\mid \calF_t}.
\end{align}
The change of measure $\P^h$ is well-known in the literature for Markov processes, see \cite{PalmowskiRoski2002technique} and references within. If the function $h$ is harmonic, i.e. $L h = 0$, it is known as \textit{Doob's h-transform}, following its introduction in \cite{Doob1984Classical}. In \cite{PalmowskiRoski2002technique} it was shown that the $X$ remains Markovian under $\P^h$ and solves the martingale problem corresponding to a perturbation of $L$.

In this article we aim to establish conditions on the $h$-function under which $X$ is not only Markovian under the changed measure, but again the mild solution of another SPDE, differing from Equation \eqref{eq: semilin_spde} by an additional drift-term dependent on $h$.
This can be viewed as a special case in which $\P^h$ is a \textit{Girsanov}-type change of measure.
In this spirit, we show the following as the main result of this paper. 

\begin{theorem}[Informal]
\label{thm: intro}
Under suitable assumptions on $h \in \dom_m(L)$, there exists a unique measure $\P^h$ on $(\Om, \calF, (\calF_t)_{t \geq 0})$ that satisfies \eqref{eq: intro_changeofmeasure}. 
Furthermore, the process
\begin{align*}
W^h(t) = W(t) - \int_0^t \sqrt{Q} \Df_x \log h(s,X(s)) \, \df s, \quad t \in [0,T],
\end{align*}
is a cylindrical Wiener process with respect to $\P^h$.
In particular, $X$ under $\P^h$ solves the SPDE
\begin{align*}
    \df X(t) = \left[ A X(t) + F(t,X(t)) + Q \Df_x \log h(t,X(t)) \right] \df t + \sqrt{Q} \df W^h(t), \quad t \in [0,T].
\end{align*}
\end{theorem}

\subsection{Approach and challenges}
Let us demonstrate our approach on how to derive Theorem \ref{thm: intro} in the case that $H = \R^d$. Equation \eqref{eq: semilin_spde} then describes a \textit{stochastic ordinary differential equation (SODE)} 
\begin{align}
\label{eq: sode}
    \df x(t) = b(t,x(t)) \df t + \sqrt{q} \df w(t),
\end{align}
where $b$ is some Lipschitz continuous function, $q$ is a symmetric positive definite matrix and $w$ is an $\R^d$-valued Wiener process.
Let $h \in C^{1,2}_b(\R_+ \times \R^d)$ be differentiable with bounded derivatives.
An application of \textit{Itô's formula} then gives that $h(t,x(t))$ is the semimartingale given by
\begin{align*}
    h(t,x(t)) = h(0,x_0) + \int_0^t L_0 h(s,x(s)) \, \df s + \int_0^t \langle \sqrt{q} \Df_x h(s,x(s)),  \df w(s) \rangle,
\end{align*}
where $L_0$ is the \textit{Kolmogorov operator} associated with equation \eqref{eq: sode}, defined via
\begin{align*}
    (L_0 h)(t,x) = \partial_t h(t,x) + \langle b(t,x), \Df_x h(t,x) \rangle + \frac{1}{2} \tr[q\Df_x^2 h(t,x)].
\end{align*}
From this, one can conclude that $h \in \dom_m(L)$ with $L h = L_0 h$ and Dynkin martingale $D^h$ given by
\begin{align*}
    D^h(t) = h(0,x_0) + \int_0^t \langle \sqrt{q} \Df_x h(s,x(s)), \df w(s) \rangle.
\end{align*}
An application of the integration by parts formula for semimartingales then shows that the local martingale $E^h$ equals the \textit{stochastic exponential} 
\begin{align*}
    E^h(t) = \exp \left(M^h(t) - \frac{1}{2} \left[ M^h \right]_t \right)
\end{align*}
of the Itô process $M^h(t) = \int_0^t \langle \sqrt{q} \Df_x \log h(s,x(s)), \df w(s) \rangle$.
Therefore, if $E^h$ is a true martingale, the \textit{Girsanov theorem} implies that the process $w^h(t) = w(t) - \int_0^t \sqrt{q} \Df_x \log h(s,x(s)) \, \df s$ is a Wiener process under $\P^h$ and that $x$ solves the SODE
\begin{align*}
        \df x(t) = b(t,x(t)) \df t + q\Df_x \log h(t,x(t)) \df t + \sqrt{q} \df w^h(t).
\end{align*}
We face two key challenges when generalizing this approach to the setting of an infinite-dimensional Hilbert space $H$.
Firstly, since $A$ is an unbounded operator on $H$, we generally cannot expect the SPDE \eqref{eq: semilin_spde} to admit a strong solution. 
This renders any direct application of Itô's formula infeasible and even for smooth functions $h$, the process $h(t,X(t))$ is in general not a semimartingale. 

We circumvent this issue by approximating $X$ by a sequence of strong solutions $X_n$ that satisfy Equation \eqref{eq: semilin_spde} when substituting $A$ by its \textit{Yosida approximations} $(A_n)_n.$ Under suitable assumptions on $h$, one can then approximate the process $h(t,X(t))$ by the sequence of semimartingales $h(t,X_n(t))$ for which Itô's formula is applicable.

Secondly, consider the Kolmogorov operator 
\begin{align*}
    (L_0 \varphi)(t,x) = \partial_t \varphi(t,x) + \langle A x + F(t,x), \Df_x \varphi(t,x)\rangle + \frac{1}{2} \tr[Q \Df_x^2 \varphi(t,x)]
\end{align*}
associated with the SPDE \eqref{eq: semilin_spde}.
In order for $L_0 \varphi$ to be a well-defined and continuous function on $\R_+ \times H$, besides the usual smoothness properties of $\varphi$, one requires that there exists a continuous extension of the mapping
\begin{align*}
    \R_+ \times \dom(A) \rightarrow \R,~ (t,x) \mapsto \langle A x, \Df_x \varphi(t,x) \rangle.
\end{align*}
This severely limits the class of functions for which $L_0 \varphi$ is a well-defined differential operator and substantial work has been done to construct suitable spaces of test functions for Kolmogorov operators in infinite dimensions, see \cite{DaPrato04Kolmogorov} and references within.

In particular, for most of our applications, the $h$-functions of interest cannot be expected to be in the domain of $L_0$. However, in \cite{manca2008kolmogorov} and \cite{manca2009fokker} it is shown that the space of \textit{exponential test functions} $\calERH$, defined as the real and imaginary parts of the functions 
\begin{align*}
       \R_+ \times H \rightarrow \R, ~(s,x) \mapsto \exp(i (\langle x, a \rangle + cs)), \quad a \in \dom(A^*), c \in \R,
\end{align*}
acts as a core for the infinitesimal generator $(L, \dom_m(L))$ with respect to $\pi$-convergence. Therefore, under the weaker assumption that $h \in \dom_m(L)$, we may approximate $h$ with a sequence of suitable test functions $(h_n)_n \subset \calERH$ for which $L_0 h_n$ remains well-defined.

\subsection{Related work and Applications}
In applications of the exponential change of measure $E^h$, one chooses a suitable $h$-function such that $X$ under $\P^h$ exhibits certain desired properties. 
A well-known application from the finite-dimensional setting is the derivation of \textit{diffusion bridges} that describe the process $x$ conditioned on hitting an endpoint $x(T) = y \in \R^d$. An application of our results lifts this to the infinite-dimensional case, thereby allowing us to derive an equation for the \textit{infinite-dimensional diffusion bridge} (or \textit{SPDE bridge}).

To the best of our knowledge, the existing literature on infinite-dimensional bridges is limited to the case where $F\equiv 0$. This ensures mild solutions are Gaussian processes, which leads to an explicit expression for the SPDE bridge, also called an \textit{Ornstein-Uhlenbeck (OU)} bridge in this case.  
\cite{Simao91} shows that an infinite system of one-dimensional OU bridges defines an OU bridge on a Hilbert space. In a more general, non-diagonal setting, \cite{GoldysMaslowski08} derive an equation for the OU bridge and apply it to study basic properties of transition semigroups for semilinear SPDEs. More recently, \cite{DinunnoPetersson23} consider a linear stochastic reaction–diffusion equation on a bounded domain where the process is conditioned on a noisy observation at time $T$.  A general framework for the spatial discretization of these bridge processes is developed. 

Our approach via the change of measure $E^h$ is more general. In the specific case that $h(t,x)=p(t,x; T,y)$, with $p$ the transition density of the process with respect to an appropriately chosen reference measure, it gives rise to the infinite-dimensional diffusion bridge that conditions the process to hit $y$ at time $T$. Our results allow for other choices of $h$, for example $h(t,x)=\tilde{p}(t,x; T,y)$, where $\tilde{p}$ is the transition density of the SPDE without the nonlinearity. The resulting process is called a \textit{guided process}, analogous to the finite-dimensional setting introduced in \cite{vdMeulenSchauer2017Guided}. The guided process is different from the conditioned process, but it mimics properties of that process, though the extra term in the drift ignores the nonlinearity. Contrary to $p$, the transition density $\tilde{p}$ is explicitly known. Therefore, the SPDE for the guided process can be numerically approximated. Taking into account the likelihood ratio of distribution of the true conditioned process with respect to the guided process, {\it weighted samples} of the conditioned process are obtained. It is exactly this approach which has been proposed in earlier works in the simpler setting of stochastic ordinary differential equations (see, for instance, \cite{vdMeulenSchauer2017Guided}, \cite{delyon2006simulation}, \cite{vdMeulenBierkensSchauer2020Simulation} and applications in \cite{mider2021continuous}). The results in this paper prove the existence of the guided process as the mild solution to a particular SPDE. 

Another application that we consider is that of \textit{forcing} the process so that its marginal distribution at time $T$ is fixed to a specified distribution. This extends the results for the finite-dimensional case considered in \cite{Baudoin2002Conditioned}.

\subsection{Outline}
We provide an overview of the needed preliminaries in Section \ref{sec: 2_preliminaries}. Particular attention will be given to semigroups that are strongly continuous with respect to $\pi$-convergence as well as their infinitesimal generators. 
In Section \ref{sec: 3_main_result}, we present, in a detailed manner, the main result of this paper and its proof. Additionally, a modified version is given for the special case that the change of measure is limited to a finite time horizon. 
We showcase some applications of the main result in Section \ref{sec: 4_applications}. These include the derivations of the infinite-dimensional diffusion bridge and the forced process as well as the guided process. 

\section{Preliminaries}
\label{sec: 2_preliminaries}

\subsection{On the stochastic basis} 
We assume to be working on the stochastic basis $(\Omega, \calF, (\calF_t)_{t \geq 0}, \P)$ defined as follows. 
Let $H$ be some Hilbert space, embedded in another Hilbert space $H'$ such that the embedding $J: H \hookrightarrow H'$ is a Hilbert-Schmidt operator. Note that in particular, $JJ^*$ is a positive definite trace-class operator on $H'$. Define $\Om = C(\R_+;H')$ as the space of all continuous functions from $\R_+$ to $H'$ equipped with the metric
\begin{align*}
    d(\om,\tilde{\om}) = \sum_{n=1}^{\infty} \frac{1}{2^n} \frac{\|\om - \tilde{\om}\|_n}{1+\|\om - \tilde{\om}\|_n}, 
\end{align*}
where $\|\om - \tilde{\om}\|_n = \sup_{t \in [0,n]}|\om(t) - \tilde{\om}(t)|.$
Then $(\Om, d)$ is a Polish space and we denote by $\calF$ the Borel $\sigma$-algebra of $(\Om, d)$.

On $(\Om, \calF)$ define the canonical process $\eta: \R_+ \times \Om \rightarrow H'$, $\eta_t( \om) = \om(t)$ and let $(\calF_t)_{t \geq 0}$ be the right-continuous extension of the natural filtration of $\eta_t$, i.e.  $\calF_t = \bigcap_{\eps > 0} \sigma( \eta_s : s \leq t+\epsilon)$ for any $t \geq 0.$ By the Kolmogorov extension theorem, there exists a Gaussian measure $\P$ on $(\Om, \calF, (\calF_t)_{t \geq 0})$ such that, under $\P$, the canonical process $\eta$ is a Wiener process on $H'$ with covariance operator $JJ^*$. In particular, $\eta$ is a cylindrical Wiener process on $H$. 

\vspace{2mm}

Denote by $\P_t$ the restriction of $\P$ onto $\calF_t$ for any $t \geq 0$. We will need the following result (see \cite{Stroock1987Lectures}, Lemma 4.2.).
\begin{lemma}
\label{lem: ExpChangeMeasure}
Let $(E(t))_{t \geq 0}$ be a non-negative martingale on $(\Omega, \calF, (\calF_t)_{t \geq 0}, \P)$ with $\E[E(0)] = 1$.
Then there exists a unique measure $\Q$ on $\calF$ such that $\df \Q_t = E(t) \df \P_t$ for all $t \geq 0$.
\end{lemma}

\begin{remark}
Note that Lemma \ref{lem: ExpChangeMeasure} does not give absolute continuity of $\Q$ with respect to $\P$. 
However, under the stronger assumption that $E$ is a uniformly integrable martingale, it can be shown that $\Q$ is absolutely continuous with respect to $\P$ on $\calF$ with $\df \Q = E_{\infty} \df \P$, where $E_{\infty}$ is the unique random variable such that $E(t) = \E \left[ E_{\infty} \mid \calF_t \right]$ for all $t \geq 0$.
\end{remark}

In many applications the martingale $E$ is only defined on some half-open interval $t \in [0,T)$. In that circumstance the following variation of Lemma \ref{lem: ExpChangeMeasure} will be useful. See Appendix \ref{app: A} for the proof.
\begin{lemma}
\label{lem: ExpChangeMeasureFinite}
Let $(E(t))_{t \in [0,T)}$ be a non-negative martingale on $(\Omega, \calF, (\calF_t)_{t \in [0,T)}, \P)$ with $\E[E(0)] = 1$. 
Then there exists a unique measure $\Q$ on $\calF_T$ such that $\df \Q_t = E(t) \df \P_t$ for all $0 \leq t < T.$
\end{lemma}

\subsection{On stochastic evolution equations}
The following is a standing assumption on the components involved in Equation \eqref{eq: semilin_spde}.
\begin{assumption} ~
\label{ass: basic_assumptions_SPDE}
    \begin{itemize}
        \item[(i)] $A$ is the generator of a strongly continuous semigroup $(S_t)_{t\geq 0}$ on $H$. In particular there exists a $C_S > 0, \om_S \in \R$ such that $\|S_t\|_{L(H)} \leq C_S \exp(\om_S t)$ for all $t \geq 0.$
        \item[(ii)] $W$ is a cylindrical Wiener process on $H$. 
        \item[(iii)] $Q$ is a symmetric, positive operator on $H$. Furthermore, the family of operators $(Q_t)_{t \geq 0}$ defined by 
        \begin{align}
        \label{eq: def_Q_t}
            Q_t = \int_0^t S_s Q S_s^* \,  \df s 
        \end{align}
        is such that $\sup_t \tr(Q_t) < \infty.$
        \item[(iv)] $F$ is such that there exists a constant $C_{F} > 0$ with 
        \begin{align*}
                    \| F(t,x) - F(t,y) \| \leq C_{F} \| x- y \| \text{ and }         \| F(t,x) \| \leq C_{F} (1 + \|x\|) \quad 
        \end{align*}
        for all $t \geq 0$ and $ x,y \in H.$
    \end{itemize}
\end{assumption}

Under Assumption \ref{ass: basic_assumptions_SPDE}, Equation \eqref{eq: semilin_spde} admits a unique mild solution $X = (X(t,s,x))_{t \geq s }$ for any initial value $x \in H$, i.e. $X$ is an $H$-valued, $\calF_t$-adapted process that satisfies  
\begin{align}
\label{eq: mild_sol}
    X(t,s,x) = S_{t-s} x + \int_s^t S_{t-u}F(u,X(u,s,x)) \,  \df u + \int_s^t S_{t-u} \sqrt{Q} \df W(u), \quad t \geq s.
\end{align}
The process $X$ is Markovian and has a $\P$-almost surely continuous modification.
Furthermore, for any $m \geq 1$, there exist some constants $C_{m} > 0, \gamma_m \geq 0$, also depending on $\om_S, C_S, C_F$ and $Q$, such that
\begin{align}
\label{eq: expectation_bound_mild_solution}
    \E[ \|X(t,s,x)\|^m ] \leq C_m \exp(\gamma_m (t-s)) (1 + \|x\|^m).
\end{align}
Though bounds similar to \eqref{eq: expectation_bound_mild_solution} are well known results in the literature, this particular bound follows from \cite{GoldysKocan01Diffusion}, Proposition 3.1.
To abbreviate notation, we fix some arbitrary $x_0 \in H$ and simply write $X(t)$ whenever \eqref{eq: semilin_spde} is assumed to be initialized at $X(0) = x_0$.

Due to the unbounded nature of $A$, one generally cannot assume Equation \eqref{eq: semilin_spde} to admit a strong solution. However, one can approximate its mild solution with strong solutions to a sequence of substitute equations in the following way. 
Let $(A_n)_{n > \omega_S}$ be the Yosida approximation of $A$, i.e. $(A_n)_{n > \omega_S}$ is the sequence of bounded, linear operators on $H$ defined via 
\begin{align*}
    A_n = n A R(n,A),
\end{align*}
where $R(n,A) = (n-A)^{-1}$ is the resolvent of $A$. It then holds that $\lim_n A_n x = A x$ for any $x \in \dom(A)$ and that $A_n$ defines a semigroup $S^{(n)}_t$ such that $\lim_n S^{(n)}_t x = S_t x$ for all $x \in H$ and 
\begin{align}
\label{eq: yosida_approx_Sn_bound}
    \|S^{(n)}_t\|_{L(H)} \leq C_S \exp \left(  \omega_n t \right)
\end{align}
with $\omega_n = \dfrac{\omega_S n}{n - \omega_S}$.
Now let $X_n$ be the strong solution to the equation
\begin{align}
\label{eq: def_Xn_yosida_approx}
\df X_n(t) = \left[A_n X_n(t) + F(t,X_n(t)) \right] \df t + \sqrt{Q} \df W(t), \quad X_n(s) = x.
\end{align}
It is well-known (see e.g. \cite{DaPrato2014Stochastic}, Proposition 7.4) that
\begin{align}
\label{eq: yosida_Xn_approx}
    \lim_{n \to \infty} \E \left[ \sup_{t \in [s,T]} |X_n(t) - X(t)|^p \right] = 0 
\end{align}
for any $T > s, p > 1$ and in particular $X_n \to X$ in probability as $C([s,T];H)$-valued random variables.

In the special case that $F = 0$ we denote the time homogeneous mild solution to Equation \eqref{eq: semilin_spde} by $Z$ and refer to it as the \textit{Ornstein-Uhlenbeck (OU) process}. The random variables $Z(t,s,x)$ are Gaussian with mean $S_{t-s} x$ and covariance operator $Q_{t-s}$. 
Under Assumption \ref{ass: basic_assumptions_SPDE} (iii) $Z(t,s,x)$ converges in distribution to its invariant distribution $\nu \sim \calN(0,Q_{\infty})$ with 
\begin{align}
\label{eq: q_infty}
    Q_{\infty} = \int_0^{\infty} S_u Q S_u^* \, \df u.
\end{align}

\subsection{On transition semigroups and their generators on spaces of polynomial growth}

Let $(E, |\cdot|_E)$ be a Banach space. 
For any $m \geq 0$, we let $C_m(\R_+ \times H;E)$ be the space of all continuous functions $\varphi: \R_+ \times H \rightarrow E$ that are bounded in $t$ and of at most polynomial growth of order $m$ in $x$.
The space $C_m(\R_+ \times H;E)$ is a Banach space with norm
\begin{align*}
    \| \varphi\|_m = \sup_{(t,x) \in \R_+ \times H} \frac{|\varphi(t,x)|_E}{1 + \|x\|^m}.
\end{align*}
If $E = \R^d$, equipped with the Euclidean norm, we simply write $C_m(\R_+ \times H)$. 

It follows from the inequality \eqref{eq: expectation_bound_mild_solution} that the time homogeneous space-time process $Y(t,(s,x)) = (t+s, X(t+s,s,x))$ defines a transition semigroup $(T_t)_{t \geq 0}$ on $C_m(\R_+ \times H)$ via 
\begin{align}
\label{eq: T_t}
    (T_t \varphi)(s,x) = \E[ \varphi(s+t, X(t+s,s,x))], \quad \varphi \in C_m(\R_+ \times H), t \geq 0, x \in H.
\end{align}
As noted in the introduction, $(T_t)_{t \geq 0}$ is not strongly continuous with respect to the norm topology on $C_m(\R_+ \times H)$. Instead, one has to turn to the weaker mode of \textit{$\pi$-convergence}.

We say a sequence $(\varphi_n)_n \subset C_m(\R_+ \times H)$ is \textit{$\pi$-convergent} to $\varphi \in C_m(\R_+ \times H)$ and write $\pilim_n \varphi_n = \varphi$ if
\begin{align*}
\sup_n \| \varphi_n \|_{m} < \infty \text{ and } \lim_n \varphi_n(t,x) = \varphi(t,x) \text{ for all } (t,x) \in \R_+ \times H.
\end{align*}

The \textit{$\pi$-closure} of a subset $B$ in $C_m(\R_+ \times H)$ is defined as 
\begin{align*}
        \overline{B}^{\pi} = \{  \varphi \in C_m(\R_+ \times H): ~ \exists (\varphi_n)_n \subset B \text{ s.t. } \pilim_n \varphi_n = \varphi\}.
\end{align*}
The set $B$ is said to be \textit{$\pi$-closed} if $\overline{B}^{\pi} = B$ and \textit{$\pi$-dense} if $\overline{B}^{\pi} = C_m(\R_+ \times H)$.
A linear operator $L: \dom(L) \subset C_m(\R_+ \times H) \rightarrow C_m(\R_+ \times H)$ is a \textit{$\pi$-closed operator} if the graph $\{ (\varphi, L \varphi) : \varphi \in \dom(L) \}$ is $\pi$-closed in $C_m(\R_+ \times H) \times C_m(\R_+ \times H)$. If a subdomain $D \subset \dom(L)$ is such that for any $\varphi \in \dom(L)$ there exists a sequence $(\varphi_n)_n$ in $D$ with
\begin{align}
\label{eq: pi_core_property}
    \pilim_n \varphi_n = \varphi \text{ and } \pilim_n L \varphi_n = L \varphi
\end{align}
we call $D$ a $\pi$-core for $(L, \dom(L)).$

The following definition of $\pi$-semigroups is based on \cite{Priola1999Markov}. They are exactly those semigroups that are `strongly continuous' with respect to $\pi$-convergence.
\begin{definition}[$\pi$-semigroup]
\label{def: pi_semigroup}
Let $(T_t)_{t \geq 0}$ be a semigroup of bounded, linear operators on $C_m(\R_+ \times H), m \geq 0$. We say $(T_t)_{t \geq 0}$ is a \textit{$\pi$-semigroup} if the following conditions hold:
\begin{itemize}
    \item[(i)] \label{item: i} There exist some $M \geq 1$ and $\om \in \R$ such that for all $t \geq 0$
    \begin{align}
    \label{eq: exp_bound_pi_semigrp}
                \|T_t \| \leq M \exp(\om t).
    \end{align}
    \item[(ii)] \label{item: ii}For any $t \geq 0$ and any $(\varphi_n)_n \subset C_m(\R_+ \times H)$ such that $\pilim_n \varphi_n = \varphi \in C_m(\R_+ \times H)$ we have 
    \begin{align}
        \pilim_n T_t \varphi_n = T_t \varphi \in C_m(\R_+ \times H).
    \end{align}
    \item[(iii)] For any $\varphi \in C_m(\R_+ \times H)$ and $(s,x) \in \R_+ \times H$ fixed, the mapping 
    \begin{align}
        [0,\infty) \longrightarrow \R, ~t \mapsto T_t \varphi(s,x) 
    \end{align} is continuous.
\end{itemize}
\end{definition}

\begin{remark}
    Note that contrary to the case of semigroups that are strongly continuous with respect to the norm topology on $C_m(\R_+ \times H)$, condition $(i)$ needs to be assumed as it does not follow from the other conditions. 
\end{remark}

\begin{remark}
Conditions $(i)$ and $(iii)$ imply that any $\pi$-semigroup $(T_t)_{t \geq 0}$ is `strongly continuous' with respect to $\pi$-convergence, i.e. if $(t_n)_n$ is a sequence such that $t_n \downarrow 0$ then 
\begin{align*}
        \pilim_n T_{t_n} \varphi = \varphi, \quad \varphi \in C_m(\R_+ \times H).
\end{align*}
We then write $\pilim_{t \downarrow 0} T_t \varphi = \varphi$.
\end{remark}

\begin{lemma}
\label{lem: Tt_is_pisemigroup}
 For any $m \geq 0$, the semigroup $(T_t)_{t \geq 0}$ defined in \eqref{eq: T_t} is a $\pi$-semigroup on $C_m(\R_+ \times H)$.
\end{lemma}
\begin{proof}
We show that $(T_t)_{t \geq 0}$ satisfies properties (i) to (iii) in Definition \ref{def: pi_semigroup}.
Let $\varphi \in C_m(\R_+ \times H)$. Then it holds for any $(s,x)$ that
\begin{align}
\label{eq: 9p8u1231fsdfj93}
\begin{split}
    |(T_t \varphi)(s,x)| &\leq \E[ | \varphi(s+t, X(s+t,s,x))  | ]  \\
    &\leq \|\varphi\|_{m} \E[ 1 + \| X(s+t,s,x) \|^m] \\
    &\leq \| \varphi\|_{m} (1 + C_m \exp(\gamma_m t) (1 + \|x\|^m)),
\end{split}
\end{align}
where we used the definition of $\|\varphi\|_{m}$ in the second and the inequality \eqref{eq: expectation_bound_mild_solution} in the third step. It follows that
\begin{align*}
    \| T_t \varphi\|_{m} &= \sup_{(s,x) \in \R_+ \times H} \frac{| T_t \varphi(s,x)|}{1 + \|x\|^m} \\
    &\leq \| \varphi\|_{m} \sup_{(s,x) \in \R_+ \times H} \frac{(1 + C_m \exp(\gamma_m t) (1 + \|x\|^m))}{1+ \|x\|^m} \\
    &\leq  \| \varphi\|_{m} (1+ C_m) \exp(\gamma_m t)
\end{align*}
and thus \eqref{eq: exp_bound_pi_semigrp} holds with $M = (1+ C_m)$ and $\om = \gamma_m$.

To show (ii), let $(\varphi_n)_n$ be a sequence in $C_m(\R_+ \times H)$ such that $\pilim_n \varphi_n = \varphi$.
It then follows by the dominated convergence theorem that
\begin{align*}
    \lim_n (T_t \varphi_n)(s,x) &= \lim_n \E[ \varphi_n(t+s, X(t+s,s,x))] \\
    &= \E[ \varphi(t+s, X(t+s,s,x))] = (T_t \varphi)(s,x).
\end{align*}
Furthermore, from \eqref{eq: 9p8u1231fsdfj93} and $\sup_n \|\varphi\|_{m} < \infty$ it follows that $\sup_n \|T_t \varphi_n\|_{m} < \infty$ and therefore that $\pilim_n T_t \varphi_n = T_t \varphi. $
To show (iii), it suffices to note that continuity of $t \mapsto T_t \varphi(s,x)$ follows from the almost sure continuity of $t \mapsto X(\cdot, s,x)$ and dominated convergence.  
\end{proof}

The \textit{infinitesimal generator} of a $\pi$-semigroup is defined in a similar manner as for a strongly continuous semigroup.
\begin{definition}
\label{def: inf_generator}
    Let $(T_t)_{t \geq 0}$ be a $\pi$-semigroup on $C_m(\R_+ \times H)$. 
    The \textit{infinitesimal generator} $L$ of $(T_t)_{t \geq 0}$ is the operator defined via 
    \begin{align}
    \label{eq: def_inf_gen}
    \begin{split}
    \begin{cases}
        \dom_m(L) &= \left\{ \varphi \in C_m(\R_+ \times H): \exists \psi \in C_m(\R_+ \times H) \text{ s.t. } \pilim_{t \downarrow 0} \dfrac{T_t \varphi - \varphi}{t} = \psi \right\} \\
        (L \varphi)(s,x) &= \lim_{t \downarrow 0} \dfrac{(T_t \varphi)(s,x) - \varphi(s,x)}{t}, \quad \varphi \in \dom_m(L), (s,x) \in \R_+ \times H.
    \end{cases}
    \end{split}
    \end{align}
\end{definition}

As the upcoming result shows, the infinitesimal generator of a $\pi$-semigroup satisfies the common properties characteristic for generators of strongly continuous semigroups.
\begin{theorem} 
\label{thm: properties_inf_gen}
   Let $L$ be the infinitesimal generator of a $\pi$-semigroup $(T_t)_{t \geq 0}$ on $C_m(\R_+ \times H)$. Then:
   \begin{itemize}
       \item[(i)] The domain $\dom_m(L)$ is $\pi$-dense in $C_m(\R_+ \times H)$.
       \item[(ii)] The operator $L$ is the unique $\pi$-closed operator such that for all $\lambda > \omega$ the resolvent $R(\lambda, L) = (\lambda - L)^{-1}$ is a bounded, linear operator on $C_m(\R_+ \times H)$ with 
   \begin{align}
   \label{eq: resolvent}
       R(\lambda, L) \varphi(s,x) = \int_0^{\infty} \exp(-\lambda t) T_t \varphi(s,x) \, \df t.
   \end{align}
      \item[(iii)] It holds that $T_t(\dom_m(L)) \subset \dom_m(L)$ and for all $\varphi \in \dom_m(L)$ and fixed $(s,x)\in \R_+ \times H$ the function $t \mapsto T_{t} \varphi(s,x)$ is differentiable with
    \begin{align}
    \label{eq: differential_property_semigroup}
        \dfrac{\df}{\df t}T_t \varphi(s,x) = L T_t \varphi(s,x) = T_t L \varphi(s,x).
    \end{align}
   \end{itemize}  
\end{theorem}
\begin{proof}
See \cite{Priola1999Markov}, Propositions 3.2 to 3.6., where these properties have been shown in the context of $\pi$-semigroups on $C_b(H)$. 
\end{proof}

The following is essentially a version of \textit{Dynkin's formula} in the context of $\pi$-semigroups.
\begin{lemma}
\label{lem: dynkins_formula}
Let $Y(t) = (t, X(t))$ be the space-time process of $X(t)$. Then for any $\varphi \in \dom(L)$ the process
\begin{align}
    D^{\varphi}(t) = \varphi(Y(t)) - \int_0^t L \varphi(Y(s)) \, \df s
\end{align}
is an $\calF_t$-adapted $\P$-martingale. We call $D^{\varphi}(t)$ the \textit{Dynkin martingale} (of $(\varphi, L \varphi)$).
\end{lemma}
\begin{proof}
    Clearly $Y(t) = (t, X(t,x))$ is $\calF_t$-adapted with initial condition $Y(0) = (0,x_0)$. It holds for any $r \leq t$ that
    \begin{align}
    \label{eq: 09234u1ndp}
    \begin{split}
    \E[  D^{\varphi}(t) \mid \calF_r] &= \E \left[ \varphi(Y(t))  - \int_0^t L \varphi(Y(u)) \, \df u \mid \calF_r \right] \\
    &= \E[ \varphi(Y(t)) \mid \calF_r] - \E \left[ \int_0^r L \varphi(Y(u))\,  \df u \mid \calF_r \right] - \E \left[ \int_r^t L \varphi(Y(u)) \, \df u \mid \calF_r \right] \\
    &= T_{t-r} \varphi(Y(r)) - \int_0^r L \varphi(Y(u)) \, \df u  -  \int_r^t T_{u-r} L \varphi(Y(r)) \, \df u ,
    \end{split}
    \end{align}
    where we used the Markov property and $\calF_t$-adaptedness of $Y(t)$ in the last step. By Theorem \ref{thm: properties_inf_gen} (iii) and a substitution we have 
    \begin{align*}
    \int_r^t T_{u-r} L \varphi(Y(r))\,  \df u  &= \int_0^{t-r} T_{u} L \varphi(Y(r))\,  \df u  \\
    &= T_{t-r} \varphi(Y(r)) - \varphi(Y(r)) \quad \P\text{-a.s.}
    \end{align*}
    Plugging this into \eqref{eq: 09234u1ndp} gives $\E[  D^{\varphi}(t) \mid \calF_r] = \varphi(Y(r))  - \int_0^r L \varphi(Y(u)) \, \df u$
    which gives the claim. 
\end{proof}

In general, the abstract definition of the infinitesimal generator given in \eqref{eq: def_inf_gen} does not lend itself easily to a closed form expression of $L \varphi$ for $ \varphi \in \dom_m(L)$.
However, in many cases one can instead construct a suitable $\pi$-core for $(L, \dom_m(L))$ on which $L$ acts in a more `descriptive' manner. 
It turns out that in this case, $L$ acts as the Kolmogorov operator associated with $\eqref{eq: semilin_spde}$ on the space of \textit{exponential test functions} $\calERH$, defined as the span of the real and imaginary parts of the functions
\begin{align}
\label{eq: exponential_test_functions}
   \R_+ \times H \rightarrow \R, ~(s,x) \mapsto \exp(i (\langle x, a \rangle + cs)), \quad a \in \dom(A^*), c \in \R,
\end{align}
and $\calERH$ defines a $\pi$-core of $(L, \dom_m(L))$ as the following lemma shows. The $\pi$-core property of $\calERH$ will play a crucial part in the proof our main result.

\begin{lemma}
\label{lem: pi_core}
For any $m \geq 1$, the space $\calERH$ is a subset of $\dom_m(L)$ with $L \varphi = L_0 \varphi$, where
\begin{align}
\begin{split}
    (L_0 \varphi)(t,x) &= \partial_t \varphi(t,x) + \langle x, A^* \Df_x \varphi(t,x) \rangle + \langle F(t,x), \Df_x \varphi(t,x) \rangle + \frac{1}{2} \tr\left(Q \Df_x^2 \varphi(t,x)\right)
\end{split}
\end{align}
for any $\varphi \in \calERH.$
Moreover, $\calERH$ is a $\pi$-core for $(L, \dom_m(L))$, i.e. for any $\varphi \in \dom_m(L)$ there exists a sequence $(\varphi_n)_n \subset \calERH$ \footnote{This approximation relies on multi-indexed sequences, see Appendix \ref{app: B}. However, for ease of notation, we may assume that the sequence only has one index. } such that
\begin{align}
    \label{eq: ERH_pi_core_property}
    \pilim_n \varphi_n = \varphi \text{ and } \pilim_n L_0 \varphi_n = L \varphi.
\end{align}
Furthermore, if $\varphi \in (L, \dom_m(L))$ is such that $\Df_x \varphi \in C_m(\R_+ \times H;H)$, the approximating sequence $(\varphi_n)_n$ in $\eqref{eq: ERH_pi_core_property}$ can be chosen such that
\begin{align}
\label{eq: picore_Dxapprox}
        \pilim_n \Df_x \varphi_n =  \Df_x \varphi.
\end{align}
\end{lemma}
\begin{proof}
    This follows from some slight generalizations of the results in \cite{manca2009fokker}, where the $\pi$-core property of exponential test functions for generators of $\pi$-semigroups was studied in the context of time homogeneous SPDEs. For more details, see Appendix \ref{app: B}.
\end{proof}

\begin{remark}
Note that for $\varphi \in \calERH$, $L \varphi$ is in general not a bounded function anymore. 
Let for example $\varphi(t,x) = \sin(\langle x, h \rangle), h \in \dom(A^*)$. 
Then $\Df_x \varphi = h \cos(\langle x, h \rangle)$ and $\Df_x^2 \varphi(x) = - h \otimes h \sin(\langle x, h \rangle)$ and thus
\begin{align*}
    L \varphi = \cos(\langle x, h \rangle) (\langle x, A^* h \rangle  + \langle F(t,x), h  \rangle) - \frac{1}{2} \sin(\langle x, h\rangle) \langle \sqrt{Q} h, \sqrt{Q}h \rangle,
\end{align*}
which is not bounded but of linear growth under the Lipschitz assumption on $F$.
In particular, $\calERH$ is not a $\pi$-core for $(L, \dom_m(L))$ if $m = 0$.
\end{remark}

\section{Main result}
\label{sec: 3_main_result}
In this section we present the proof of our main result. It establishes conditions on the $h$-function for which, under the exponential change of measure $\P^h$, the mild solution $X$ to the SPDE \eqref{eq: semilin_spde} is a mild solution to yet another SPDE with an additional drift term dependent on the $h$-function of choice.

Our point of departure is that, following Lemma \ref{lem: dynkins_formula}, for any $h \in \dom_m(L)$, the process
\begin{align*}
    D^h(t) = h(t,X(t)) - \int_0^t L h(s,X(s)) \, \df s 
\end{align*}
is a $\P$-martingale. Additionally, it can be shown that for any positive function $h \in \dom_m(L)$, the process 
\begin{align}
\label{eq: E^h(t)}
    E^h(t) = \dfrac{h(t,X(t))}{h(0,x_0)} \exp\left( - \int_0^t \dfrac{L h}{h}(s,X(s)) \, \df s \right), \quad t \geq 0,
\end{align}
is a continuous $\P$-local martingale whenever it exists, see for example \cite{PalmowskiRoski2002technique}, Lemma 3.1.
If $E^h$ is a true $\P$-martingale, it follows from Lemma \ref{lem: ExpChangeMeasure} that it defines a unique measure $\P^h$ on $\calF$ such that
\begin{align}
\label{eq: changeofmeasure}
    \dfrac{\df \P_t^{h}}{\df \P_t} =  E^h(t), \quad t > 0.
\end{align}

Throughout this section we fix some arbitrary $m \geq 1$. We will need the following assumptions.
\begin{assumption} 
\label{ass: hgoodfunction}
The function $h: \R_+ \times H \longrightarrow \R_{>0}$ satisfies:
\begin{itemize} 
\item[(i)] $h \in \dom_m(L)$ such that $h^{-1} Lh \in C_m(\R_+ \times H)$.
\item[(ii)] $h$ is Fréchet differentiable in $x$ such that $\Df_x h \in C_m(\R_+ \times H;H)$. 
\item[(iii)] $h$ is such that $E^h$ is a $\P$-martingale.
\end{itemize}
\end{assumption}

The following is this sections main result.
\begin{theorem}
\label{thm: weaksolution}
    Let $h$ satisfy Assumption \ref{ass: hgoodfunction} and let $\P^h$ be the measure defined by \eqref{eq: changeofmeasure}. Then, for any $T > 0$, the process 
    \begin{align}
    \label{eq: W^h}
    W^h(t) = W(t) - \int_0^t \sqrt{Q} \Df_x \log h(s,X(s)) \df s, \quad t \in [0,T],
    \end{align}
    is a cylindrical Wiener process with respect to $\P^h$.
    In particular, $X$ under $\P^h$ is a mild solution to the SPDE
    \begin{align}
        \df X(t) = \left[ A X(t) + F(t,X(t)) + Q \Df_x \log h(t,X(t)) \right] \df t + \sqrt{Q} \df W^h(t), \quad t \in [0,T].
    \end{align}
\end{theorem}

The main part of the proof of Theorem \ref{thm: weaksolution} lies in the following lemma. 
\begin{lemma}
\label{lem: hDynkinMartingaleItoProcess}
    Let $h$ satisfy Assumption \ref{ass: hgoodfunction} (i) and (ii) and let $D^h$ be the corresponding Dynkin martingale.
    Then 
    \begin{align}
    \label{eq: D^h_is_Ito}
        D^h(t) = h(0,x_0) + \int_0^t \bigl \langle \sqrt{Q} \Df_x h(s,X(s)), \df W(s) \bigr \rangle,  \quad t \geq 0.
    \end{align}
    Furthermore, let $E^h$ be the process as defined in \eqref{eq: E^h(t)}. Then $E^h$ is the stochastic exponential given by
    \begin{align}
    \label{eq: E^h_is_ExpMartingale}
        E^h(t) = \exp \left(M^h(t) - \frac{1}{2} [M^h]_t \right), \quad t \geq 0,
    \end{align} where $M^h(t) = \int_0^t \langle \sqrt{Q} \Df_x \log h(s,X(s)), \df W(s) \rangle$.
\end{lemma}
\begin{proof}
\textit{Step One.}
We begin by showing the claim in Equation \eqref{eq: D^h_is_Ito} for any exponential test function $h \in \calERH$. 

Let $(A_n)_{n > \om_S}$ be the Yosida approximation of $A$ such that $\lim_n A_n x = A x$ for all $x \in \dom(A)$ and let $X_n$ be the sequence of strong solutions to equation \eqref{eq: def_Xn_yosida_approx}.
Define the sequence of processes $(D_n^h)_n$ via
\begin{align}
\label{eq: sodfj89243}
    D_n^h(t) = h(t,X_n(t)) - \int_0^t Lh(s,X_n(s)) \, \df s.
\end{align}
We first show that $D_n^h(t)\xrightarrow{\P} D^h(t)$ for any $t \geq 0$.
It follows from \eqref{eq: yosida_Xn_approx} that $h(t,X_n(t)) \xrightarrow{\P}  h(t,X(t))$ and $Lh(t,X_n(t)) \xrightarrow{\P}  Lh(t,X(t))$.
Furthermore, it holds that
\begin{align}
\begin{split}
\label{eq: adfj89213}
   \sup_n \E \left[ \left|  \int_0^t Lh(s,X_n(s))\,  \df s \right| \right] &\leq \sup_n \E \left[   \int_0^t  \frac{ \left| Lh(s,X_n(s)) \right|}{1+\|X_n(s)\|^m}  (1+\|X_n(s)\|^m)   \, \df s \right] \\
    &\leq \|L h\|_{m} \sup_n   \int_0^t  (1+ \E[\|X_n(s)\|^m ])  \, \df s  \\
    &< \infty.
\end{split}
\end{align}
Here, we used in the last step that, by the bounds in \eqref{eq: expectation_bound_mild_solution} and \eqref{eq: yosida_approx_Sn_bound}, there exists for any $m \geq 1$ some $\tilde{C}_m > 0$ and $\tilde{\gamma}_m \geq 0$ such that 
\begin{align}
\label{eq: yosida_approx_Xn_uniform_expectation_bound}
    \sup_n \E[\|X_n(s)\|^m ] \leq \tilde{C}_m \exp( \tilde{\gamma}_m s).
\end{align}
By the dominated convergence theorem it follows from \eqref{eq: adfj89213} that 
\begin{align*}
    \int_0^t Lh(s,X_n(s)) \, \df s \xrightarrow{L^1(\P)}  \int_0^t Lh(s,X(s))\,  \df s
\end{align*} and therefore in total we get that $D_n^h(t)\xrightarrow{\P} D^h(t)$. 

On the other hand, for any $n > \om_S$, an application of Itô's lemma gives that
\begin{align*}
    \df h(t,X_n(t)) = L_n h(t,X_n(t)) \df t + \langle \sqrt{Q} \Df_x h(t,X_n(t)), \df W(t) \rangle,
\end{align*}
where $L_n h(t,x) = \partial_t h(t,x) + \langle x, A_n^* \Df_x h(t,x) \rangle + \langle F(x), \Df_x h(t,x) \rangle + \frac{1}{2} \tr[Q \Df_x^2 h(t,x)]$ is the Kolmogorov operator associated with the approximating SPDE in \eqref{eq: def_Xn_yosida_approx}.
Plugging this into \eqref{eq: sodfj89243} gives that
\begin{align}
\label{eq: ajp293923}
\begin{split}
    D_n^h(t) &= h(0,x_0) + \int_0^t \bigl \langle \sqrt{Q} \Df_x h(s,X_n(s)), \df W(s) \bigr \rangle + \int_0^t \left[ L_n h(s,X_n(s)) - L h(s,X_n(s)) \right]\,  \df s \\
    &= h(0,x_0) + \int_0^t \bigl \langle \sqrt{Q} \Df_x h(s,X_n(s)), \df W(s) \bigr \rangle +  \int_0^t  \bigl \langle X_n(s), \left(A_n^* - A^*\right) \Df_x h(s,X_n(s)) \bigr \rangle \,  \df s.
\end{split}
\end{align}

It remains to show that 
\begin{align}
\label{eq: adf09u243}
    \lim_n \int_0^t \bigl \langle \sqrt{Q} \Df_x h(s,X_n(s)), \df W(s) \bigr \rangle = \int_0^t \bigl \langle \sqrt{Q} \Df_x h(s,X(s)), \df W(s) \bigr\rangle
\end{align}
and 
\begin{align}
\label{eq: ajf8903}
    \lim_n  \int_0^t  \bigl \langle X_n(s), \left(A_n^* - A^*\right) \Df_x h(s,X_n(s)) \bigr\rangle \, \df s = 0
\end{align}
in probability. Then \eqref{eq: D^h_is_Ito} follows from $D_n^h(t)\xrightarrow{\P} D^h(t)$ and \eqref{eq: ajp293923}-\eqref{eq: ajf8903}.

To show \eqref{eq: adf09u243}, note that by the Itô isometry and \eqref{eq: yosida_approx_Xn_uniform_expectation_bound} we have
\begin{align*}
    \sup_n \E \left[ \left\Vert  \int_0^t \bigl \langle \sqrt{Q} \Df_x h(s,X_n(s)), \df W(s) \bigr \rangle \right\Vert^2 \right] &= \sup_n \int_0^t \E \left[ \| \sqrt{Q} \Df_x h(s,X_n(s)) \|^2 \,  \df s \right] \\ &\leq \sup_n \|Q\| \|\Df_x h\|^2_{m} \int_0^t  \E[ ( 1 + \|X_n(s)\|^{m})^2] \, \df s \\
    &< \infty.
\end{align*}
Since $\Df_x h(s,X_n(s)) \xrightarrow{\P} \Df_x h(s,X(s))$, we get \eqref{eq: adf09u243} by an application of the dominated convergence theorem.

Lastly, for any $h \in \calERH$, it holds that $\Df_x h(s,x) = \sum_{i=1}^k z_i g_i(s,x)$ for some $z_i \in \dom(A^*)$ and $g_i \in \calERH$, $i = 1,...k.$ 
We show \eqref{eq: ajf8903} for $k=1$. The case that $k > 1$ follows from linearity. 
We have
\begin{align*}
    \int_0^t  \bigl \langle X_n(s), \left(A_n^* - A^*\right) \Df_x h(s,X_n(s)) \bigr \rangle \, \df s &= \int_0^t  \bigl \langle X_n(s), \left(A_n^* - A^*\right) z_1 \bigr \rangle g_1(s,X_n(s))\,  \df s.
\end{align*}
By \eqref{eq: yosida_Xn_approx} it follows that $X_n \xrightarrow{\P} X$ in $C([0,t];H)$ and in particular $\sup_n \sup_{s \in [0,t]} \|X_n(s)\| < \infty$ a.s.
Thus \eqref{eq: ajf8903} follows from $\lim_n \left(A_n^* - A^*\right) z_1 = 0$. In total this concludes that \eqref{eq: D^h_is_Ito} holds for any $h \in \calERH.$

\textit{Step Two.}
We proceed to show \eqref{eq: D^h_is_Ito} for any $h \in \dom_m(L)$ that satisfies Assumptions \ref{ass: hgoodfunction} (i) and (ii). By Lemma \ref{lem: pi_core} there exists a sequence $(h_n)_n \in \calERH$ such that 
\begin{align}
\label{eq: ansdf8921}
    \pilim_n h_n = h, ~\pilim_n \Df_x h_n = \Df_x h \text{ and } \pilim_n L h_n = L h. 
\end{align}
Denote by $D^{h_n}(t)$ the Dynkin martingale of $(h_n, L h_n)$.
Firstly, it follows by dominated convergence, the bound in \eqref{eq: expectation_bound_mild_solution} and the $\pi$-convergence in \eqref{eq: ansdf8921}
that
\begin{align}
\begin{split}
\label{eq: a902834f}
    D^{h_n}(t) &= h_n(t,X(t)) - \int_0^t L h_n(s,X(s)) \, \df s \\
    &\xrightarrow[]{\P}  h(t,X(t)) - \int_0^t L h(s,X(s))\,  \df s \\
    &= D^h(t).
\end{split}
\end{align}

On the other hand, since $(h_n)_n \subset \calERH$, it follows from Step One that
\begin{align}
    D^{h_n}(t) &= h_n(0,x_0) + \int_0^t \bigl \langle \sqrt{Q} \Df_x h_n(s,X(s)), \df W(s) \bigr \rangle, \quad n \in \N.
\end{align}
We already have that $\lim_n h_n(0,x_0) = h(0,x_0)$. It thus remains to show that
\begin{align}
\label{eq: ap90123nbf}
     \int_0^t \bigl \langle \sqrt{Q} \Df_x h_n(s,X(s)), \df W(s) \bigr \rangle &\xrightarrow[]{\P} \int_0^t \langle \sqrt{Q} \Df_x h(s,X(s)), \df W(s) \rangle.
\end{align}
 By Itô isometry it suffices to establish that
\begin{align*}
     \lim_{n \to \infty} \E \left[ \int_0^t \|\sqrt{Q} \Df_x h_n(s,X(s))\|^2 \, \df s \right] =  \E \left[ \int_0^t \|\sqrt{Q} \Df_x h(s,X(s))\|^2 \, \df s \right],
\end{align*}
but this again follows from an application of the dominated convergence theorem, the $\pi$-convergence of $\pilim_n \Df_x h_n = \Df_x h$ and the bound in \eqref{eq: expectation_bound_mild_solution}. In total, the claim \eqref{eq: D^h_is_Ito} thus follows from \eqref{eq: a902834f} - \eqref{eq: ap90123nbf}.

\textit{Step Three.}
We finish the proof by showing the claim in \eqref{eq: E^h_is_ExpMartingale}. 
First note that by Assumption \ref{ass: hgoodfunction} (i) and the bound in \eqref{eq: expectation_bound_mild_solution}, the process
\begin{align}
\label{eq: aj2p8934p9}
    t \mapsto \int_0^t \dfrac{L h}{h}(s,X(s)) \, \df s 
\end{align}
is finite $\P$-a.s. and the process $E^h$ in \eqref{eq: E^h(t)} is therefore well-defined with $\E \left[ E^h(0) \right]  =1 $. Furthermore, the mapping in \eqref{eq: aj2p8934p9} is continuous and of finite variation and thus (see e.g. \cite{PalmowskiRoski2002technique}, Lemma 3.1)
\begin{align}
    \left[ h(t,X(t)), \exp \left(- \int_0^t \frac{L h}{h}(s,X(s))\,  \df s \right) \right]_t = 0.
\end{align}
By Lemma \ref{lem: dynkins_formula} the process $h(t,X(t))$ is a semimartingale and with the integration by parts formula for semimartingales it follows that 
\begin{align}
\begin{split}
\label{eq: ap92384}
    \df E^h(t) &= \frac{1}{h(0,x_0)} \left[  h(t,X(t)) \df \left( \exp \left(- \int_0^t \frac{L h}{h}(s,X(s)) \, \df s \right) \right)+ \exp \left(- \int_0^t \frac{L h}{h}(s,X(s))\,  \df s \right) \df h(t,X(t)) \right] \\ 
    &= \frac{1}{h(0,x_0)} \left[-  \exp \left(- \int_0^t \frac{L h}{h}(s,X(s)) \, \df s \right)  Lh(t,X(t)) \df t + \exp \left(- \int_0^t \frac{L h}{h}(s,X(s)) \, \df s \right) \df h(t,X(t)) \right] \\
    &= \frac{1}{h(0,x_0)} \exp \left(- \int_0^t \frac{L h}{h}(s,X(s)) \, \df s \right) \df D^h(t),
\end{split}
\end{align}
where $D^h(t) = h(t,X(t)) - \int_0^t Lh(s,X(s)) \, \df s$ is the Dynkin martingale of $(h, Lh)$. 
From step two it follows that $\df D^h(t) = \langle \sqrt{Q} \Df_x h(t,X(t)), \df W(t) \rangle$ and plugging this into \eqref{eq: ap92384} gives
\begin{align}
\begin{split}
    \df E^h(t) &=  \frac{1}{h(0,x_0)} \exp \left(- \int_0^t \frac{L h}{h}(s,X(s)) \, \df s \right) \langle \sqrt{Q} \Df_x h(t,X(t)), \df W(t) \rangle  \\
    &= E^h(t) \langle \sqrt{Q} \Df_x \log h(t,X(t)), \df W(t) \rangle.
\end{split}
\end{align}
In particular, the process $E^h(t)$ is the stochastic exponential $E^h(t) = \exp \left( M^h(t) - \frac{1}{2}[M^h]_t \right)$ of the Itô process 
\begin{align}
    M^h(t) = \int_0^t \bigl \langle \sqrt{Q} \Df_x \log h(s,X(s)), \df W(s) \bigr \rangle.
\end{align}
\end{proof}

We are now ready to give the remainder of the proof of Theorem \ref{thm: weaksolution}, which is essentially just an application of Girsanov's theorem. 
\begin{proof}[Proof of Theorem \ref{thm: weaksolution}.]
Let $E^h$ be the process as defined in \eqref{eq: E^h(t)}. 
By Assumption \ref{ass: hgoodfunction} (iii), $E^h$ is a $\P$-martingale and thus, following Lemma \ref{lem: ExpChangeMeasure}, defines a measure $\P^h$ on $\calF$ such that $\df \P^h_T = E^h(T) \df \P_T$ on $\calF_T$ for any $T > 0$.
On the other hand, Lemma \ref{lem: hDynkinMartingaleItoProcess} shows that $E^h$ is the stochastic exponential of the Itô process 
\begin{align}
    M^h(t) = \int_0^t \bigl \langle \sqrt{Q} \Df_x \log h(s,X(s)), \df W(s) \bigr \rangle.
\end{align}

It therefore follows from Girsanov's theorem that
\begin{align}
        W^h(t) = W(t) - \int_0^t \sqrt{Q} \Df_x \log h(s,X(s)) \, \df s, \quad t \in [0,T],
\end{align}
is a cylindrical Wiener process with respect to the measure $\P^h$. In particular, under $\P^h$, the process $X$ given in $\eqref{eq: mild_sol}$ is a mild solution to the SPDE
\begin{align*}
 \df X(t) = \left[ A X(t) + F(t,X(t)) + Q \Df_x \log h(t,X(t)) \right] \df t + \sqrt{Q} \df W^h(t), \quad t \in [0,T].
\end{align*}

\end{proof}

\subsection{The change of measure on a finite time interval}

In certain applications we work with $h$-functions that are only defined on the half open interval $[0,T)$ for some $T > 0$. 
We then replace Assumption \ref{ass: hgoodfunction} with the following.

\begin{assumption}
\label{ass: hTgoodfunction}
The function $h: [0,T) \times H \longrightarrow \R_+$ satisfies for any $S < T$:
\begin{itemize}
\item[(i)] $h \in C_m([0,S] \times H)$ such that $L h(t,x)$ exists for any $[0,S] \times H$ and $h^{-1} Lh \in C_m([0,S] \times H)$.
\item[(ii)] $h$ is Fréchet differentiable in $x$ such that $\Df_x h \in C_m([0,S] \times H;H)$. 
\item[(iii)] $h$ is such that $(E^h(t))_{t \in [0,S]}$ is a $\P$-martingale.
\end{itemize}
\end{assumption}

Under Assumption \ref{ass: hTgoodfunction} (iii), $(E^h(t))_{t \in [0,T)}$ is a non-negative martingale with $\E[E(0)] = 1$. Following Lemma \ref{lem: ExpChangeMeasureFinite}, it thus defines a measure $\P^h$ on $\calF_T$ such that 
\begin{align}
\label{eq: FinTimechangeofmeasure}
    \dfrac{\df \P_t^{h}}{\df \P_t} =  E^h(t), \quad t \in [0,T).
\end{align}

We get the following version of Theorem \ref{thm: weaksolution}.
\begin{theorem}
\label{thm: FinTimeWeakSol}
    Let $h$ satisfy Assumption \ref{ass: hTgoodfunction} and let $\P^h$ be the measure defined on $\calF_T$ by \eqref{eq: FinTimechangeofmeasure}. Then for any $S < T$, the process 
    \begin{align}
    \label{eq: FinTimeW^h}
    W^h(t) = W(t) - \int_0^t \sqrt{Q} \Df_x \log h(s,X(s)) \,  \df s, \quad t \in [0,S],
    \end{align}
    is a cylindrical Wiener process with respect to $\P^h$.
    In particular, $X$ under $\P^h$ is a mild solution to the SPDE
    \begin{align}
        \df X(t) = \left[ A X(t) + F(t,X(t)) + Q \Df_x \log h(t,X(t)) \right] \df t + \sqrt{Q} \df W^h(t), \quad t \in [0,T).
    \end{align}
\end{theorem}

\begin{proof}
    Let $S < T$ be arbitrary but fixed. Let $\bar{h}$ be an extension of $h$ onto $\R_+ \times H$ defined via
    \begin{align*}
    \bar{h}(t,x) = \begin{cases}
        h(t,x), &\quad t \leq S, \\
        h(S,x), &\quad t > S.
    \end{cases}
    \end{align*}
    Then $\bar{h}$ satisfies Assumption \ref{ass: hgoodfunction}. Following the proof of Lemma \ref{lem: hDynkinMartingaleItoProcess}, we get that
    \begin{align*}
        E^h(t) = \exp \left(M^h(t) - \frac{1}{2} [M^h]_t\right), \quad t \in [0,S],
    \end{align*}
    where $M^h(t) = \int_0^t \langle \sqrt{Q} \Df_x \log h(s,X(s)), \df W(s) \rangle$.
    The claim then follows from Assumption \ref{ass: hTgoodfunction} (iii) and an application of Girsanov's theorem. 
\end{proof}

Typically, the martingale property of $E^h$ in Assumption \ref{ass: hTgoodfunction} (iii) is the most difficult of the three to verify. The following lemma summarizes conditions under which it is satisfied. 
\begin{lemma}
\label{lem: conditions_E^h_martingale}
Either of the following conditions are sufficient for $(E^h(t))_{t \in [0,S]}$ to be a martingale:
\begin{itemize}
    \item[1.] In addition to Assumption \ref{ass: hTgoodfunction} (i), $h^{-1} Lh  \in C_0([0,S] \times H)$ is bounded.
    \item[2.] In addition to Assumption \ref{ass: hTgoodfunction} (i) and (ii), it holds that
    \begin{align*}
        \E \left[ \exp \left( \frac{1}{2} \int_0^S \| \sqrt{Q} \Df_x \log h(t,X(t)) \|^2 \df t \right) \right] < \infty.
    \end{align*}
    \item[3.] In addition to Assumption \ref{ass: hTgoodfunction} (i) and (ii), the mapping $(t,x) \mapsto \sqrt{Q} \Df_x \log h(t,x)$ is Lipschitz in $x$, uniformly in $t \in [0,S]$.
\end{itemize}

\end{lemma}

\begin{proof}
    The first condition has been shown in \cite{PalmowskiRoski2002technique}, Proposition 3.2.
    The second condition is the well-known Novikov condition, see for example \cite{DaPrato2014Stochastic}, Proposition 10.17.
    A proof for the last condition can be found in Appendix \ref{app: C}.
\end{proof}

\section{Applications}
\label{sec: 4_applications}

\subsection{Conditioned SPDEs}
In this subsection we introduce a class of $h$-functions for which the change of measure $\P^h$ corresponds to conditioning the process $(X(t))_{t \in [0,T)}$ on $X(T)$. Applications of this type of conditioning include for example the infinite-dimensional diffusion bridge.

For the construction of the $h$-functions in this section, we rely on the transition density of $X$ with respect to some suitable reference measure. 
Since the state space of $X$ is not Euclidean, the typical choice of the Lebesgue measure is not available to us. However, one can construct a suitable Gaussian reference measure on $(H, \calB(H))$ as follows. 

In addition to Assumption \ref{ass: basic_assumptions_SPDE}, let the following hold.
\begin{assumption}[Strong Feller assumption]
\label{ass: OU_strong_feller}
The semigroup $(S_t)_{t \geq 0}$ and covariance operators $(Q_t)_{t \geq 0}$ defined in \eqref{eq: def_Q_t} are such that
\begin{align*}
    \im(S_t) \subset \im(Q_t^{1/2}), \quad t \geq 0.
\end{align*}
\end{assumption}
Under Assumption \ref{ass: OU_strong_feller}, the Ornstein-Uhlenbeck process $Z$ is strongly Feller and absolutely continuous with respect to its invariant measure $\nu \sim \calN(0,Q_{\infty})$ with covariance operator $Q_{\infty}$ as defined in \eqref{eq: q_infty}.

Furthermore, it follows from the Girsanov theorem that $\bbL(X(t,s,x)) \sim \bbL(Z(t,s,x))$ for all $t \geq s$.  
In particular, we have that $\bbL(X(t,s,x)) \sim \nu$ and we define by $p(s,x;t,y)$ the transition density 
\begin{align}
\label{eq: transition_density}
    p(s,x;t,y) = \dfrac{\df \bbL(X(t,s,x))}{\df \nu}(y), \quad \nu\text{-a.e. } y \in H,
\end{align}
of $X$ with respect to $\nu$.
From the Markov property of $X$ it follows that $p(s,x;t,y)$ satisfies the \textit{Chapman-Kolmogorov} equation
\begin{align}
\label{eq: Chapman-Kolmogorov}
    p(s,x;t,y) = \int_H p(s,x;r,z) p(r,z;t,y) \, \nu(\df z)
\end{align}
for all $s < r <t$ and for $\nu$-a.e. every $y \in H.$ Let us now define the function $h: [0,T) \times H \rightarrow \R_{>0}$ via
\begin{align}
\label{eq: hMuFunction}
    h(t,x) = \int_H p(t,x;T,y) \, \mu(\df y),
\end{align}
for some measure $\mu$ on $(H, \calB(H))$ such that $\int_H p(0,x_0;T,y) \, \mu(\df y) < \infty.$ The $h$-function constructed in \eqref{eq: hMuFunction} satisfies the following.
\begin{lemma}
\label{lem: hMuSpaceTimeHarmonic}
The function $h$ is \textit{space-time harmonic}, i.e. for any $(s,x) \in [0,T) \times H$ it holds that
\begin{align*}
    (T_t h)(s,x) = h(s,x), \quad t < T-s.
\end{align*}
In particular, $Lh = 0$ and $h(t,X(t))$ is a $\P$-martingale. 
\end{lemma}
\begin{proof}
Let $(s,x) \in [0,T) \times H$ be fixed. Then for any $t < T-s$ we have
\begin{align*}
    (T_t h )(s,x) &= \E \left[ h(t+s, X(t+s,s,x)) \right] \\
    &= \int_H h(t+s,z) p(s,x;t+s,z) \, \nu(\df z) \\
    &= \int_H \int_H p(t+s,z;T,y) p(s,x;t+s,z) \, \nu(\df z) \mu(\df y) \\
    &= \int_H  p(s,x;T,y) \, \mu(\df y) \\
    &= h(s,x),
\end{align*}
where we use the definition of $h$ in the third and the Chapman-Kolmogorov equation the fourth line. 
By definition of $L$ it follows that $Lh = 0$. 
Furthermore, using the Markov property of $(t,X(t))$ we have for any $s < t < T$ that
\begin{align*}
    \E[ h(t,X(t)) \mid \calF_s] = (T_{t-s} h)(s,X(s)) = h(s,X(s)).
\end{align*}
\end{proof}

From Lemma \ref{lem: hMuSpaceTimeHarmonic} it follows that $h$ satisfies Assumption \ref{ass: hTgoodfunction} (i) and (iii).
In particular, $h(t,X(t))/h(0,x_0)$ is a non-negative martingale with mean one and thus defines a unique measure $\P^h$ on $\calF_T$ via
\begin{align}
\label{eq: a90122p9834}
    \frac{\df \P_t^h}{\df \P_t} = \frac{h(t,X(t))}{h(0,x_0)}, \quad t < T.
\end{align}
On the other hand, Assumption \ref{ass: hTgoodfunction} (ii) is hard to verify in general, as the Fréchet differentiability of $h$ depends on $p(s,x;t,y)$ as well as the choice of $\mu$. We thus keep it as an assumption.
\begin{assumption}
\label{ass: hMuGoodFunction}
The function $h$ defined in \eqref{eq: hMuFunction} is Fréchet differentiable in $x$ such that $\Df_x h \in C_m([0,S] \times H;H)$ for any $S < T$. 
\end{assumption}

The following result shows how the measure $\P^h$ changes the law of $X$.
\begin{proposition}
Let $\P^h$ be the measure defined by \eqref{eq: a90122p9834}.
It then holds for any bounded and measurable function $g$ and $0 \leq t_1 \leq ... \leq t_n < T$ that
\begin{align}
\label{eq: a09234dsj}
    \E^h[g(X(t_1),...,X(t_n))] = \int_H \E[ g(X(t_1),...,X(t_n)) \mid X(T) = y]  \, \xi(\df y),
\end{align}
where $\xi$ is the measure defined on $(H, \calB(H))$ via
\begin{align}
    \xi(\df y) = \frac{p(0,x_0;T,y) \mu(\df y)}{\int_H p(0,x_0;T,y) \, \mu(\df y)}.
\end{align}
Additionally, if Assumption \ref{ass: hMuGoodFunction} is satisfied, $X$ under $\P^h$ satisfies the SPDE
\begin{align*}
    \df X(t) = \left[ A X(t) + F(t,X(t)) + Q \Df_x \log h(t,X(t)) \right] \df t + \sqrt{Q} \df W^h(t), \quad t \in [0,T),
\end{align*}
where $W^h$ is the $\P^h$-cylindrical Wiener process as defined in \eqref{eq: FinTimeW^h}.
\end{proposition}
\begin{proof}
To show \eqref{eq: a09234dsj} it suffices to show that
\begin{align*}
        \E^h[g(X(t))] = \int_H \E[ g(X(t)) \mid X(T) = y] \,  \xi(\df y)
\end{align*}
for any $t < T$ and continuous and bounded $g:H \rightarrow \R$. The claim then follows by a standard cylindrical argument, see e.g. \cite{EthierKurtz2009markov}, Proposition 4.1.6.
Indeed it holds that
\begin{align*}
    \E^h[ g(X(t))] &= \E \left[  g(X(t)) \frac{h(t,X(t))}{h(0,x_0)} \right] \\
    &= \int_H  g(x) \frac{h(t,x)}{h(0,x_0)} p(0,x_0;t,x) \,  \nu(\df x) \\
    &= \int_H  g(x) \frac{ \int_H p(t,x;T,y) \, \mu(\df y)}{ \int_H p(0,x_0;T,y) \, \mu(\df y)} p(0,x_0;t,x) \, \nu(\df x) \\
    &= \int_H \left(\int_H g(x) \frac{p(0,x_0;t,x) p(t,x;T,y)}{p(0,x_0;T,y)}  \, \nu(\df x) \right) \frac{p(0,x_0;T,y)}{\int_H p(0,x_0;T,y)\,  \mu(\df y)} \, \mu(\df y)\\
    &= \int_H \E[ g(X(t)) \mid X(T) = y] \,  \xi(\df y)
\end{align*}
with $\xi(\df y) = \frac{p(0,x_0;T,y) \mu(\df y)}{\int_H p(0,x_0;T,y) \, \mu(\df y)}.$ Here we used the definition of $h$ in the third, Fubini's theorem in the fourth and Bayes' theorem in the last step.
The second claim of the proposition is a direct consequence of Theorem \ref{thm: FinTimeWeakSol} under Assumption \ref{ass: hMuGoodFunction}.
\end{proof}

The formula 
\begin{align*}
        \E^h[g(X(t_1),...,X(t_n))] = \int_H \E[ g(X(t_1),...,X(t_n)) \mid X(T) = y] \,  \xi(\df y)
\end{align*}
provides a disintegration of the conditioned process: to draw from it one first generates an endpoint $X(T) = y$ from $\xi(\df y)$, followed by drawing the path conditioned on this value of $y$, see Example \ref{ex: bridge} below.
Different choices of the measure $\mu$ in \eqref{eq: hMuFunction} correspond to different kinds of conditioning. We illustrate this in the following examples.

\begin{example}[The infinite-dimensional diffusion bridge]
\label{ex: bridge}
Let $y \in H$ be such that $p(t,x;T,y)$ is well-defined. \footnote{This holds for $\nu$-a.e. $y \in H$.} Set $\mu = \delta_{y}$ as the Dirac measure
\begin{align*}
        \delta_{y}(A) = \begin{cases}
                    1, & \text{if } y \in A  \\
                     0, & \text{else}.
             \end{cases}
\end{align*}
It follows that $\xi = \delta_{y}$ and thus the formula \eqref{eq: a09234dsj} reduces to 
\begin{align*}
    \E^h[g(X(t_1),...,X(t_n))] = \E[ g(X(t_1),...,X(t_n)) \mid X(T) = y].
\end{align*}
In other words, $X$ under $\P^h$ is the process $(X(t))_{t \in [0,T)}$ conditioned on hitting the endpoint $X(T) = y$. We refer to this process as the \textit{infinite-dimensional diffusion bridge}. 
If the transition density $p(t,x;T,y)$ satisfies Assumption \ref{ass: hMuGoodFunction}, the infinite-dimensional diffusion bridge is characterized by the bridge equation
\begin{align}
\label{eq: dXstr}
    \df \Xstr(t) = \left[ A \Xstr(t) + F(t,\Xstr(t)) + Q \Df_x \log p(t,\Xstr(t); T,y) \right] \df t + \sqrt{Q} \df W(t).
\end{align}
This equation corresponds to the well known diffusion bridge equation for the case of a finite-dimensional state space. 
\end{example}

\begin{example}[Conditioning on a noisy observation]
    Suppose we do not observe the endpoint $X(T) = y$ directly, but instead we observe a sample $v$ from a distribution $q(v \mid y) \nu( \df v)$ where $q(\cdot \mid y)$ is some probability density function with respect to $\nu$. This corresponds to observing $X(T)$ under noise $q$. Furthermore, assume $q$ is such that
    \[ \mu(\df y) = q(v\mid y) \nu(\df y) \]
    is a finite measure. 
    It then follows that 
    \[ \xi(\df y) = \frac{p(0,x_0; T,y) q(v\mid y) \nu(\df y)}{\int p(0,x_0; T,y)  q(v\mid y) \, \nu(\df y)}. \]
    This has a nice Bayesian interpretation where the endpoint $y$  gets  assigned prior density $\pi(y)=p(0,x_0; T,y)$ and the observation is given by $v$.  The likelihood for this observation is  $\ell(y\mid v)=q(v\mid y)$ and hence
    \[ \xi(\df y) =\frac{\pi(y) \ell(y\mid v)  \nu(\df y)}{\int \pi(y) \ell(y\mid v) \, \nu(\df y)}.\]
    This shows that 
    $\xi(\df y)$ gives the posterior measure of  $y$, conditional upon observing $v$. Therefore, using this $h$-transform, the conditioned process is constructed by first sampling the endpoint $y$ conditional on the observation, followed by sampling the bridge to $y$. 
\end{example}

\begin{example}[The forced/tilted process]
Let $q$ be some arbitrary density function with respect to $\nu$ such that   
\begin{align*}
    \mu(\df y) = \frac{q(y) }{ p(0,x_0;T,y)} \nu(\df y)
\end{align*}
defines a finite measure on $(H, \calB(H)).$
By straightforward substitution we get that
\begin{align*}
    h(t,x) = \int_H \frac{p(t,x;T,y)}{p(0,x_0;T,y)} q(y) \,  \nu(\df y)
\end{align*}
and $\xi( \df y) = q(y) \nu (\df y)$. Hence this corresponds to \textit{forcing/tilting} $X(T)$ to have the distribution $q(y) \nu(\df y).$
\end{example}

\subsection{Guided processes}
In Example \ref{ex: bridge}, setting $h(t,x) = p(t,x;T,y)$, we have derived the infinite-dimensional diffusion bridge equation 
\begin{align*}
    \df \Xstr(t) = \left[ A \Xstr(t) + F(t,\Xstr(t)) + Q \Df_x \log p(t,\Xstr(t); T,y) \right] \df t + \sqrt{Q} \df W(t), \quad t \in [0,T),
\end{align*}
that characterizes the law of the conditioned process $X(t) \mid X(T) = y$ for $\nu$-a.e. $y \in H$. 
In many practical applications, one seeks to draw samples from $\Xstr$. In general, however, the transition density $p(t,x;T,y)$ is not tractable, rendering a direct simulation of $\Xstr$ infeasible. 

This motivates the following construction. Let $\ptld(t,x;T,y)$ be a tractable density of the mild solution $\Xtld$ of another auxiliary SPDE. 
Then, setting $\htld(t,x) = \ptld(t,x;T,y)$, one can define a change of measure $\P^{\htld}$ such that
\begin{itemize}
    \item[(i)] $X$ under $\P^{\htld}$ equals in law the mild solution $\Xcrc$ to the SPDE
    \begin{align}
    \label{eq: Xcrc}
                \df \Xcrc(t) = \left[ A \Xcrc(t) + F(t,\Xcrc(t)) + Q \Df_x \log\htld(t,\Xcrc(t)) \right] \df t + \sqrt{Q} \df W(t), \quad t \in [0,T).
    \end{align}
    \item[(ii)] $\P^h$ and $\P^{\htld}$ - and therefore, the laws of $\Xstr$ and $\Xcrc$ on $C([0,T];H)$ - are absolutely continuous with some likelihood ratio $\Phi.$
\end{itemize}
Samples of the bridge process $\Xstr$ can then be obtained by drawing proposal samples from the law of $\Xcrc$ and accepting or rejecting the proposals based on the likelihood ratio $\Phi$. We now showcase the idea of this construction in more detail.

For this, let $\ptld$ denote the transition density of the OU process $Z$ with respect to $\nu$, i.e. 
\begin{align*}
    \ptld(s,x;t,y) = \frac{\df \bbL(Z(t,s,x))}{\df \nu}(y), \quad \nu\text{-a.e. } y \in H, s < t.
\end{align*}
Since $Z$ is a Gaussian process and $\nu$ a Gaussian measure on $(H, \calB(H))$, the densities $\ptld(s,x;t,y)$ can be obtained by the Cameron-Martin formula. 
We shall place an additional assumption on the covariance operators of $Z$.
\begin{assumption}
\label{ass: Q_t_commute}
    The covariance operators $(Q_t)_{t \geq 0}$ as defined in \eqref{eq: def_Q_t} commute. 
\end{assumption}

The following proposition shows that $\ptld$ defines a function $\htld$ that satisfies Assumption \ref{ass: hTgoodfunction} and that the SPDE induced by the changed measure $\P^{\htld}$ remains tractable.
\begin{proposition}
For $\nu$-a.e. $y \in H$, the function $\htld(t,x) = \ptld(t,x;T,y)$ satisfies Assumption \ref{ass: hTgoodfunction}.
Moreover,
\begin{align}
    \Df_x \log \htld(t,x) = \Gamma^*_{T-t} Q^{-\fsqrt}_{T-t} [y - S(T-t) x],
\end{align}
where $\Gamma_{T-t} = Q^{-\fsqrt}_{T-t} S_{T-t}$ is a bounded operator on $H$ by Assumption \ref{ass: OU_strong_feller}.
In particular, there exists a unique change of measure $\P^{\htld}$ such that $X$ under $\P^{\htld}$ is a mild solution to the SPDE
\begin{align}
            \df \Xcrc(t) = \left[ A \Xcrc(t) + F(t,\Xcrc(t)) + Q  \Gamma^*_{T-t} Q^{-\fsqrt}_{T-t} [y - S(T-t) \Xcrc(t)] \right] \df t + \sqrt{Q} \df W^{\htld}(t), \quad t \in [0,T).
\end{align}
\end{proposition}

\begin{proof}
Let $S < T$ be arbitrary but fixed. 
We start by showing that Assumption \ref{ass: hTgoodfunction} (ii) is satisfied, i.e. that $\htld(t,x)$ is Fréchet differentiable in $x$ on $[0,S] \times H$ with derivative $\Df_x \htld(t,x)$ of at most polynomial growth in $x$.
For this, write
\begin{align*}
    \ptld(t,x;T,y) = \dfrac{\df \bbL(Z(T,t,x))}{\df \bbL(Z(T,t,0))}(y) \dfrac{\df \bbL(Z(T,t,0))}{\df \nu}(y), \quad \nu\text{-a.e. } y \in H.
\end{align*}

Since $Z$ is a Gaussian process, the Cameron-Martin formula gives that
\begin{align*}
    \dfrac{\df \bbL(Z(T,t,x))}{\df \bbL(Z(T,t,0))}(y) &= \exp \left( \bigl \langle Q_{{T-t}}^{-\fsqrt} y, Q_{{T-t}}^{-\fsqrt} S_{{T-t}} x \bigr \rangle - \fsqrt \| Q_{{T-t}}^{-\fsqrt} S_{{T-t}} x \|^2 \right) \\
    &=  \exp \left( \bigl \langle \Gamma^*_{{T-t}} Q_{{T-t}}^{-\fsqrt} y, x \bigr \rangle - \fsqrt \| \Gamma_{{T-t}} x \|^2 \right), \quad \nu\text{-a.e. } y \in H,
\end{align*}
where the mapping $t \mapsto \Gamma^*_{{T-t}} Q_{{T-t}}^{-\fsqrt} y$ is continuous for $\nu$-a.e. $y \in H$, see Appendix \ref{app: D} for details. 
It follows that, for $\nu$-a.e. $y \in H$ fixed, the function $\htld(t,x) = \ptld(t,x;T,y)$ is well-defined and bounded on $[0,S] \times H$.
Moreover, it is Fréchet differentiable in $x$ with derivative
\begin{align*}
  \Df_x \htld(t,x) = \htld(t,x) \Gamma^*_{{T-t}} \left(Q_{{T-t}}^{-\fsqrt} y - \Gamma_{{T-t}} x \right).
\end{align*}
It holds that $(S_t)_t$ and $(Q_t)_t$ are strongly continuous, from which it follows that $t \mapsto \Gamma_{T-t} x$ is continuous for any fixed $x \in H.$
In particular, by the uniform boundedness principle and the continuity of $t \mapsto \Gamma^*_{{T-t}} Q_{{T-t}}^{-\fsqrt} y$, we have that
\begin{align*}
     \sup_{t \in [0,S]} \left( \| \Gamma_{T-t} \| + \|\Gamma^*_{{T-t}} Q_{{T-t}}^{-\fsqrt} y\|\right)< \infty.
\end{align*}
In total we get that $\Df_x \htld \in C_1([0,S] \times H;H)$ since
\begin{align*}
    \sup_{t \in [0,S], x \in H} \frac{\|\Df_x \htld(t,x)\|}{1 + \|x\|} \leq  \sup_{t \in [0,S], x \in H}  |\htld(t,x)| \left( \frac{\| \Gamma^*_{{T-t}} Q_{{T-t}}^{-\fsqrt} y \|}{1 + \|x\|} +  \frac{\|\Gamma^*_{{T-t}} \Gamma_{{T-t}} x\|}{1 + \|x\|} \right) < \infty.
\end{align*}
We proceed to show that $\htld$ satisfies Assumption \ref{ass: hTgoodfunction} (i). By the Fréchet differentiability of $\htld$, it follows from \cite{manca2009fokker}, Theorem 4.1, that
\begin{align*}
    (L \htld)(t,x) = \tilde{L} \htld(t,x) + \bigl \langle F(t,x), \Df_x \htld(t,x) \bigr \rangle,
\end{align*}
where $\tilde{L}$ is the infinitesimal generator of the Ornstein-Uhlenbeck $\pi$-semigroup. By the same arguments as in Lemma \ref{lem: hMuSpaceTimeHarmonic}, one shows that $\htld$ is harmonic with respect to $\tilde{L}$, i.e. $\tilde{L}\htld(t,x) = 0$ for any $t < T$. 
With the Lipschitz continuity of $F$, it thus follows that $(L \htld)(t,x) \in C_m([0,S] \times H)$ for any $m \geq 2$ with
\begin{align*}
    (L \htld)(t,x) =  \bigl \langle F(t,x), \Df_x \htld(t,x) \bigr \rangle.
\end{align*}
Likewise, we have that $\htld^{-1}L \htld \in C_m([0,S] \times H)$ for any $m \geq 2$ since
\begin{align*}
    (\htld^{-1}L \htld)(t,x) &= \bigl \langle F(t,x), \Df_x \log \htld(t,x) \bigr \rangle \\
    &= \Bigl \langle F(t,x), \Gamma^*_{{T-t}} \left(Q_{{T-t}}^{-\fsqrt} y - \Gamma_{{T-t}} x \right) \Bigr \rangle.
\end{align*}
Therefore, $\htld$ satisfies Assumption \ref{ass: hTgoodfunction} (i) and $E^{\htld}$ is a well-defined continuous $\P$-local martingale given by
\begin{align}
\label{eq: a29034}
    E^{\htld}(t) = \frac{\htld(t,X(t))}{\htld(0,x_0)} \exp \left( - \int_0^t \Bigl \langle F(s,X(s)), \Gamma^*_{{T-s}} \left(Q_{{T-s}}^{-\fsqrt} y - \Gamma_{{T-s}} X(s) \right) \Bigr \rangle \, \df s \right). 
\end{align}
It remains to show that $E^{\htld}$ is a true $\P$-martingale. However, this follows directly from the uniform Lipschitz continuity of $(t,x) \mapsto \sqrt{Q} \Df_x \log\htld(t,x)$ and Lemma \ref{lem: conditions_E^h_martingale}.
\end{proof}

By construction we immediately get the following. 
\begin{corollary}
\label{cor: abs_continuity_Ft}
The measures $\P^h$ and $\P^{\htld}$ are absolutely continuous on $\calF_t, t < T$, with likelihood ratio
\begin{align*}
    \dfrac{\df \P_t^h}{\df \P_t^{\htld}}(X) = \frac{h(t,X(t))}{\htld(t,X(t))} \frac{\htld(0,x_0)}{h(0,x_0)}  \exp \left( \int_0^t  \Bigl \langle F(s,X(s)), \Gamma^*_{{T-s}} \left(Q_{{T-s}}^{-\fsqrt} y - \Gamma_{{T-s}} X(s) \right) \Bigr \rangle \, \df s \right).
\end{align*}
\end{corollary}
\begin{proof}
    This follows from plugging in the likelihood functions $\df \P^h_t = E^h(t) \df \P_t$ and $\df \P^{\htld}_t = E^{\htld}(t) \df \P_t$ as given in \eqref{eq: a29034} into 
    $\df \P_t^h/\df \P_t^{\htld} = \df \P_t^h/\df \P_t ~\df \P_t/\df \P^{\htld}_t.$
\end{proof}

\begin{remark}
Corollary \ref{cor: abs_continuity_Ft} shows the absolute continuity of the measures $\P^h$ and $\P^{\htld}$ on any $\calF_t, t < T$. In other words, the measures are equivalent as long as we `stay away from the conditioning time $T$'. However, in order to draw samples from the bridge process $\Xstr$, we require absolute continuity of $\P^h$ and $\P^{\htld}$ on $\calF_T$, i.e. on the complete interval $[0,T]$.

To demonstrate this, note that $\P^h_t \sim \P_t$ for any $t < T$, but under $\P$ the event $\{X(T) = y\}$ has measure zero, meaning that samples under $\P$ will almost surely not hit the endpoint $X(T) = y.$

In contrast to this, the additional drift term $Q \Df_x \log \ptld(t,x;T,y)$ in equation \eqref{eq: Xcrc} forces the process $\Xcrc$ to hit the conditioning point $\Xcrc(T) = y$, resulting in absolute continuity of the laws of process $\Xstr$ and $\Xcrc$ on the complete interval. 
We postpone the proof of this result, along with numerical illustrations, to an upcoming article, as it is beyond the scope of this paper.
\end{remark}

\section*{Acknowledgements}
We thank professor B. Goldys (University of Sydney) for providing the arguments in Lemma \ref{applemma: Gammat_HS} and his comments on the regularity of the OU transition density.

\bibliographystyle{apalike}  
\bibliography{references}

\begin{appendices}

\section{}
\label{app: A}

\begin{lemma}[Lemma \ref{lem: ExpChangeMeasureFinite} above]
\label{app: ExpChangeMeasureFinite}
Let $(E(t))_{t \in [0,T)}$ be a non-negative martingale on $(\Omega, \calF, (\calF_t)_{t \in [0,T)}, \P)$ with $\E[E(0)] = 1$. 
Then there exists a unique measure $\Q$ on $\calF_T$ such that $\df \Q_t = E(t) \df \P_t$ for all $0 \leq t < T.$
\end{lemma}
\begin{proof}

\textit{Uniqueness.} 
By left-continuity of $ (\calF_t)_t$ we have that $\calF_T = \sigma( \bigcup_{t < T} \calF_t)$.
Thus, if $\Q$ and $\tilde{\Q}$ are two measures on $\calF_T$ such that $\Q_t = \tilde{\Q}_t$ for all $t < T$ it immediately follows that $\Q = \tilde{\Q}$.

\textit{Existence.}
For any $n \in \N$, let $t_n = T - \frac{1}{n}$ and let $\eta_n(\om)(t) = \om(t \wedge t_n)$ be the stopped canonical process on $(\Omega, \calF, (\calF_t)_{t \in [0,T)}, \P)$.
Furthermore, let $\{ \Q_n\}$ be the measures on $\calF_T$ defined by
\begin{align}
    \Q_n(A) = \E \left[ E_{t_n} \mathbbm{1}_{\eta_n^{-1}(A)} \right], \quad A \in \calF_T.
\end{align}
Noting that $\eta_{n+1}^{-1}(A) = \eta_{n}^{-1}(A)$ for any $A \in \calF_{t_n}$ it holds that
\begin{align}
\label{eq: afm0092304}
\begin{split}
    \Q_{n+1}(A) 
    &= \E \left[ \E \left[ E_{t_{n+1}} \mathbbm{1}_{\eta_{n}^{-1}(A)} \mid \calF_{t_n}\right] \right] \\
    &= \E \left[ \mathbbm{1}_{\eta_{n}^{-1}(A)} \E \left[ E_{t_{n+1}}  \mid \calF_{t_n}\right] \right] \\
    &= \Q_n(A),
\end{split}
\end{align}
i.e. $\Q_{n+1 \mid \calF_{t_n}} = \Q_{n \mid \calF_{t_n}}$ for all $n \in \N$.

Denote by $\mathcal{S}$ the algebra $\mathcal{S} = \bigcup_{n \geq 1} \calF_{t_n}$.
By left-continuity of $(\calF_t)_t$ it holds that $\calF_T = \sigma(\mathcal{S})$ and from \eqref{eq: afm0092304} it follows that the mapping
\begin{align*}
    \Q(A) = \lim_n \Q_n(A), \quad A \in \mathcal{S},
\end{align*}
is well-defined and a pre-measure on $\calS$. It therefore follows from the Carathéodory extension theorem that $\Q$ extends uniquely to a measure on $\calF_T$ and it is straightforward to check that $\Q$ satisfies $\Q(A) = \E[ E_t \mathbbm{1}_{A}]$ for any $A \in \calF_t, t < T.$

\end{proof}

\section{}
\label{app: B}
We give more details on the approximation properties of the exponential test functions $\calERH$ with respect to $\pi$-convergence in $C_m(\R_+ \times H)$.
For this, it does not suffice to consider single-indexed sequences. Hence, the results stated in this section rely on using $k$-indexed sequences, i.e. sequences $(\varphi_{n_1,...,n_k})_{n_1,...,n_k}$ that depend on $k$ indices for some $k \in \N.$

A $k$-indexed sequence $(\varphi_{n_1,...,n_k})_{n_1,...,n_k} \subset C_m(\R_+ \times H)$ is said to be $\pi$-convergent to $\varphi \in C_m(\R_+ \times H)$ if for any $i \in \{2,...,k\}$ there exists an $(i-1)$-indexed sequence $(\varphi_{n_1,...,n_{i-1}})_{n_1,...,n_{i-1}} \subset C_m(\R_+ \times H)$ such that
\begin{align*}
    \varphi_{n_1,...,n_{i-1}} = \pilim_{n_{i}} \varphi_{n_1,...,n_i}
\end{align*}
and $\varphi = \pilim_{n_1} \varphi_{n_1}.$ We then write $\pilim_{n_1,...,n_k} \varphi_{n_1,...,n_k} = \varphi$.

The following is a slight generalization of \cite{manca2008kolmogorov}, Proposition 2.7.
\begin{lemma}
\label{app: ERH_approx_property}
Let $\varphi \in C_0(\R_+ \times H)$ such that $\Df_x \varphi \in C_0(\R_+ \times H; H)$.
Then there exists a three-indexed sequence $(\varphi_{n_1,n_2,n_3})_{n_1,n_2,n_3} \subset \calERH$ such that 
\begin{align}
\begin{split}
\label{eq: ERH_approx_property}
    \pilim_{n_1,n_2, n_3} \varphi_{n_1,n_2,n_3} &= \varphi, \\
    \pilim_{n_1,n_2,n_3} \Df_x \varphi_{n_1,n_2,n_3} &= \Df_x \varphi.
\end{split}
\end{align}
\end{lemma}
\begin{proof}
We give a rough sketch only.

\textit{Step One.} We first show \eqref{eq: ERH_approx_property} for the case that $H = \R^d$ and the case that $Ax = x$. We denote by $\calE(\R_+ \times H)$ the space of exponential test functions as defined in \eqref{eq: exponential_test_functions} for the identity operator $Ax = x$.

For any $n \in \N$, let $\psi_n \in C_0(\R_+ \times \R^d)$ be such that 
\begin{itemize}
    \item[(i)] $\psi_n = \varphi$ and $\Df_x \psi_n = \Df_x \varphi$ on $(0,2n) \times (-n,n)^{d}$,
    \item[(ii)] $\psi_n$ and $\Df_x \psi_n$ are $2n$-periodic in any direction,
    \item[(iii)] $\|\psi_n\|_{0} \leq \|\varphi\|_{0}$ and $\|\Df_x \psi_n\|_0 \leq \|\Df_x \varphi \|_0$.
\end{itemize}

Then, for any fixed $n \in \N$, let $(\psi_{n,m})_{n,m}$ be the sequence given by
\begin{align*}
    \psi_{n,m}(t,x) &= \left(\frac{1}{2n}\right)^{d+1} \int_{[0,2n] \times [-n,n]^d} \psi_n(t-s, x_1 - y_1, ..., x_d-y_d) F_{n,m}(y_1,...,y_d) \,  \df s \df y_1 ... \df y_d \\
    &= (\psi_{n} \ast F_{n,m})(t,x),
\end{align*}
where $F_{n,m}$ is the $m$-th order Fejér kernel of period $2n$. 
It can be shown that $\psi_{n,m}$ equals the Cesàro mean of the $m$-th partial Fourier sum of $\psi_n$ and hence $(\psi_{n,m})_{n,m} \subset \calE(\R_+ \times \R^d)$.

Note that $\| \psi_{n,m}\|_0 \leq \| \psi_n \|_0$ for all $n,m \in \N.$
Furthermore, since $\Df_x \psi_n$ is continuous and bounded, it holds by standard properties of the convolution operator that
\begin{align*}
    \Df_x \psi_{n,m}(t,x) =  (\Df_x \psi_{n} \ast F_{n,m})(t,x)
\end{align*}
with $\| \Df_x \psi_{n,m}\|_0 \leq \| \Df_x \psi_n \|_0$.
It follows from Fejér's theorem that 
\begin{align*}
    \lim_m \| \psi_{n,m} - \psi_n \|_0 &= 0, \\
    \lim_m \|  \Df_x \psi_{n,m} - \Df_x \psi_n\|_0 &= 0.
\end{align*}
By diagonalization of $(\psi_{n,m})_{n,m}$\footnote{For example, for any $n \in \N$ let $\varphi_n = \psi_{n,m(n)}$ where $m(n)$ is such that $\| \psi_{n,m} - \psi_n \|_0 < \frac{1}{n}$ and $\|  \Df_x \psi_{n,m} - \Df_x \psi_n\|_0 < \frac{1}{n}$.} we find a sequence $(\varphi_n)_n \subset \calE(\R_+ \times \R^d)$ that satisfies \eqref{eq: ERH_approx_property}.

\textit{Step Two.} We keep the assumption that $A x = x$ but now assume that $H$ is infinite-dimensional with some orthonormal basis $(e_j)_j$. Let $P_{n_1} x = \sum_{j=1}^{n_1} \langle x, e_j \rangle e_j$ be the orthogonal projection onto $H^{n_1} = \LH\{e_j : j = 1,...,n_1\}$. 
Then, for any $n_1$ fixed and $\varphi \in C_0(\R_+ \times H)$ consider the function 
\begin{align*}
    \varphi_{n_1}(t,x) = \varphi(t, P_{n_1} x).
\end{align*}

By the first step above there exists a two-indexed sequence $(\varphi_{n_1,n_2})_{n_1,n_2}$ such that $\pilim_{n_2} \varphi_{n_1,n_2} = \varphi_{n_1}$ and $\pilim_{n_2} \Df_x  \varphi_{n_1,n_2} = \Df_x \varphi_{n_1}$.
In particular, from $\lim_{n_1} \varphi_{n_1}(t, x) = \varphi(t,x)$ and 
\begin{align*}
    \lim_{n_1} (\Df_x \varphi_{n_1})(t,x) &=  \lim_{n_1} P_{n_1} (\Df_x \varphi)(t,P_{n_1} x)  \\
    &=  \Df_x \varphi(t,x)
\end{align*}
as well as $\|\varphi_{n_1}\|_0 < \|\varphi\|_0$ and $\| \Df \varphi_{n_1} \|_0 \leq \| \Df \varphi \|_0$ it follows that \eqref{eq: ERH_approx_property} holds.

\textit{Step Three. } Now let $H$ be infinite-dimensional and $A$ be the generator of a $C_0$-semigroup on $H$ as in Assumption \ref{ass: basic_assumptions_SPDE}.
For any $n > \om_S$ let $R(n,A) = (n - A)^{-1}$ be the resolvent of $A$ such that $\lim_n n R(n,A) x = x$ and $R(n,A)x \in \dom(A)$ for any $x \in H$.

Let $\varphi$ be as assumed in the lemma. By step two there exists a sequence $(\varphi_{n_1,n_2})_{n_1,n_2} \subset \calE(\R_+ \times H)$ such that $\pilim_{n_1,n_2} \varphi_{n_1,n_2} = \varphi$ and $\pilim_{n_1,n_2} \Df_x \varphi_{n_1,n_2} = \Df_x \varphi$.

Then, setting 
\begin{align*}
    \varphi_{n_1,n_2,n_3}(t,x) = \varphi_{n_1,n_2}(t, n_3 R(n_3,A)x)
\end{align*}
it holds that $(\varphi_{n_1,n_2,n_3})_{n_1,n_2,n_3} \subset \calERH$ and it is straightforward to show that \eqref{eq: ERH_approx_property} holds.
\end{proof}

\begin{lemma}
\label{app: ERH_approx_property_Cm}
Let $\varphi \in C_m(\R_+ \times H)$ be such that $\Df_x \varphi \in C_m(\R_+ \times H; H)$.
Then there exists a four-indexed sequence $(\varphi_{n_1,n_2,n_3,n_4})_{n_1,n_2,n_3,n_4} \subset \calERH$ such that 
\begin{align}
\begin{split}
\label{eq: ERH_approx_property_Cm}
    \pilim_{n_1,n_2, n_3, n_4} \varphi_{n_1,n_2,n_3, n_4} &= \varphi, \\
    \pilim_{n_1,n_2,n_3, n_4} \Df_x \varphi_{n_1,n_2,n_3, n_4} &= \Df_x \varphi.
\end{split}
\end{align}
\end{lemma}

\begin{proof}
Let $\varphi \in C_m(\R_+ \times H)$ be such that $\Df_x \varphi \in C_m(\R_+ \times H; H)$. 
For any $n_1 \in \N$, define 
\begin{align*}
    \varphi_{n_1}(t,x) = \frac{n_1 \varphi(t,x)}{n_1+\|x\|^{2m}}.
\end{align*}
Then $\lim_{n_1} \varphi_{n_1}(t,x) = \varphi(t,x)$ pointwise and $\sup_{n_1} \|\varphi_{n_1}\|_{m} \leq \| \varphi\|_{m}$.
Furthermore, one can show that $\lim_{n_1} \Df_x \varphi_{n_1}(t,x) = \varphi(t,x)$ pointwise with $\sup_{n_1} \|\Df_x \varphi_{n_1}\|_{m} \leq (\|\varphi\|_{m} + \| \Df_x \varphi\|_{0,m })$ and in particular it holds that
\begin{align*}
    \pilim_{n_1} \varphi_{n_1} = \varphi,
    \pilim_{n_1} \Df_x \varphi_{n_1} = \Df_x \varphi.
\end{align*}

Noting that $\varphi_{n_1} \in C_0(\R_+ \times H)$ with bounded derivative $\Df_x \varphi_{n_1} \in C_0(\R_+ \times H;H)$, 
the claim follows from Lemma \ref{app: ERH_approx_property} by approximating $\varphi_{n_1}$ with a suitable sequence $(\varphi_{n_1,n_2,n_3,n_4})_{n_1,n_2,n_3,n_4} \subset \calERH$ such that
\begin{align*}
    \pilim_{n_1,n_2, n_3, n_4} \varphi_{n_1,n_2,n_3, n_4} &= \varphi_{n_1}, \\
    \pilim_{n_1,n_2,n_3, n_4} \Df_x \varphi_{n_1,n_2,n_3, n_4} &= \Df_x \varphi_{n_1}.
\end{align*}
\end{proof}

\begin{lemma}[Lemma \ref{lem: pi_core} above]
\label{lem: pi_core_app}
For any $m \geq 1$, the space $\calERH$ is a subset of $\dom_m(L)$ with $L \varphi = L_0 \varphi$, where
\begin{align}
\begin{split}
    (L_0 \varphi)(t,x) &= \partial_t \varphi(t,x) + \langle x, A^* \Df_x \varphi(t,x) \rangle + \langle F(t,x), \Df_x \varphi(t,x) \rangle + \frac{1}{2} \tr\left(Q \Df_x^2 \varphi(t,x)\right)
\end{split}
\end{align}
for any $\varphi \in \calERH.$
Moreover, $\calERH$ is a $\pi$-core for $(L, \dom_m(L))$, i.e. for any $\varphi \in \dom_m(L)$ there exists a sequence $(\varphi_n)_n \subset \calERH$ such that
\begin{align}
    \label{eq: ERH_pi_core_property_app}
    \pilim_n \varphi_n = \varphi \text{ and } \pilim_n L_0 \varphi_n = L \varphi.
\end{align}
Furthermore, if $\varphi \in (L, \dom_m(L))$ is such that $\Df_x \varphi \in C_m(\R_+ \times H;H)$, the approximating sequence $(\varphi_n)_n$ in $\eqref{eq: ERH_pi_core_property_app}$ can be chosen such that
\begin{align}
\label{eq: picore_Dxapprox_app}
        \pilim_n \Df_x \varphi_n =  \Df_x \varphi.
\end{align}
\end{lemma}
\begin{proof}
Using Lemma \ref{app: ERH_approx_property_Cm}, the proof goes just as in \cite{manca2009fokker}, Theorem 1.3 after noting that $Y(t,(s,x))$ is the mild solution to the noise-degenerate SPDE
\begin{align}
\label{eq: spacetime_semilin_spde}
	\begin{split}
		\begin{cases}
			\df Y(t) &= \left[  \tilde{A} Y(t) + \tilde{F}(Y(t)) \right] \df t + \sqrt{\tilde{Q}} \df \tilde{W}(t), \quad t \geq 0, \\
			Y(0) &= (s,x).
		\end{cases}
	\end{split}
\end{align}
Here $\tilde{A}(s,x) = (0,A x)$ and $\tilde{F}(s,x) = (1, F(s,x))$ and $\tilde{Q} = (0,Q)$ are defined on the product space $\R \times H$ and it is straightforward to show that $\tilde{A}, \tilde{F}$ and $\tilde{Q}$ satisfy Assumption \ref{ass: basic_assumptions_SPDE}.    
\end{proof}

\section{}
\label{app: C}
\begin{lemma}
    Let $G(s,x)$ be a Lipschitz continuous function in $x$, uniformly on $[0,T]$, i.e. there exists some constant $C > 0$ such that
    \begin{align*}
        \| G(s,x) - G(s,y) \| \leq C \| x- y\|
    \end{align*}
    for all $s \in [0,T]$ and $x,y \in H$.
    Let $X$ be the mild solution to \eqref{eq: semilin_spde}. Then the process $E(t)$ defined by
    \begin{align*}
        E(t) = \exp \left( \int_0^t \langle G(s,X(s)), \df W(s) \rangle - \frac{1}{2} \int_0^t \| G(s,X(s)) \|^2 \, \df s \right), \quad t \in [0,T],
    \end{align*}
    is a $\P$-martingale. 
\end{lemma}

\begin{proof}
    Since $(E(t))_{t \in [0,T]}$ is a supermartingale, it suffices to show that $\E[E(T)] = 1$. From the almost sure continuity of $X$ and the Lipschitz continuity of $G$ it follows that
    \begin{align*}
        \int_0^T \| G(s,X(s)) \|^2 \, \df s < \infty \quad \P\text{-a.s.}
    \end{align*}
    Thus, defining the stopping times 
    \begin{align*}
        \tau_n(X) = \inf \left\{ t \in [0,T] : \int_0^t \| G(s,X(s)) \|^2 \, \df s \geq n \right\} \wedge T
    \end{align*}
    we have that $\P \left( \lim_n \tau_n = T \right) = 1$. 
    Since 
    \begin{align*}
        \int_0^{T \wedge \tau_n} \| G(s,X(s)) \|^2 \, \df s < n \quad \P\text{-a.s.}
    \end{align*}
    it follows from the Novikov condition (see e.g. \cite{DaPrato2014Stochastic},  Proposition 10.17) that $E_n(t) = E(t \wedge \tau_n), t \in [0,T],$ is a $\P$-martingale for each $n \in \N.$
    In particular, $E_n$ defines a measure $\P_n$ on $\calF_T$ such that $\df \P_{n} = E_n(T) \df \P$.
    Define $\chi_n(s) = \mathbbm{1}_{s \leq \tau_n}$. 
    Noting that 
    \begin{align*}
        E_n(t) &= \exp \left( \int_0^{t \wedge \tau_n} \langle G(s,X(s)), \df W(s) \rangle - \frac{1}{2} \int_0^{t \wedge \tau_n} \| G(s,X(s))\|^2 \, \df s  \right) \\
        &= \exp \left( \int_0^{t} \langle \chi_n(s) G(s,X(s)), \df W(s) \rangle - \frac{1}{2} \int_0^{t} \| \chi_n(s) G(s,X(s))\|^2 \, \df s  \right),
    \end{align*}
    it follows from the Girsanov theorem that 
    \begin{align*}
        W_n(t) = W(t) - \int_0^t \chi_n(s) G(s,X(s)) \, \df s, \quad t \in [0,T],
    \end{align*}
    is a $\P_n$-cylindrical Wiener process. 
    It follows that for any $n \in \N$, $X$ under $\P_n$ is a mild solution to the equation
    \begin{align*}
        \df X_n = \left[ A X_n + F(t,X_n(t)) + \sqrt{Q} \chi_n(t) G(t,X_n(t)) \right] \df t + \sqrt{Q} \df W_n(t).
    \end{align*}
    In particular, by the Lipschitz continuity of $G$, there exists a $\P_{n_0}$-a.s. continuous version of $X$ for any $n_0 \in \N$ fixed, from which we conclude that
    \begin{align*}
         \int_0^T \| G(s,X(s)) \|^2 \df s < \infty \quad \P_{n_0}\text{-a.s.}
    \end{align*}

    It follows, using the monotonicity of $\tau_n$ in $n$ in the second line and monotone convergence in the last step, that
    \begin{align*}
        1 = \lim_n \P_{n_0} \left( \tau_n = T \right) &= \lim_n \int_{\{\tau_n = T\}} E(T \wedge \tau_{n_0}) \, \df \P \\
        &= \lim_{n \geq n_0} \int_{\{\tau_n = T\}} E(T) \, \df \P \\
        &= \E[E(T)].
    \end{align*}
\end{proof}

\section{}
\label{app: D}

\begin{lemma} 
\label{applemma: Gammat_HS}
Under Assumption \ref{ass: OU_strong_feller}, for any $S < T$, the operators $\Gamma_{T-t}$ are uniformly Hilbert-Schmidt on $[0,S]$, i.e.
    \begin{align*}
        \sup_{t \in [0,S]} \|\Gamma_{T-t}\|_{HS} < \infty.
    \end{align*}
\end{lemma}
\begin{proof}
From the Strong Feller assumption \ref{ass: OU_strong_feller} it follows that $\Gamma_r = Q^{-\fsqrt}_r S_r$ is a bounded linear operator and thus
\begin{align*}
    S_r = Q_r^{\fsqrt} \Gamma_r
\end{align*}
is a Hilbert-Schmidt operator for any $r > 0.$
Moreover, it follows from Assumption \ref{ass: OU_strong_feller} that 
\begin{align*}
    \im(Q_{\infty}^{\fsqrt}) = \im(Q_r^{\fsqrt}), \quad r > 0,
\end{align*}
see Proposition 2 in \cite{MichalikGoldys96regularity}.
From this one concludes that $Q_r^{-\fsqrt} Q^{\fsqrt}_{\infty}$ and $Q^{-\fsqrt}_{\infty} S_r$ are bounded linear operators for all $r > 0$.
Now, fix some arbitrary $S < T$. Then for any $t \in [0,S]$ it holds that
\begin{align*}
    \Gamma_{T-t} &= Q^{- \fsqrt}_{T-t} S_{T-t} \\
    &= (Q^{-\fsqrt}_{T-t} Q^{\fsqrt}_{\infty}) (Q^{-\fsqrt}_{\infty} S_{T-t}) \\
    &= (Q^{-\fsqrt}_{T-t} Q^{\fsqrt}_{\infty}) (Q^{-\fsqrt}_{\infty} S_{T+S-t}) S_{T-S}.
\end{align*}
Then, noting that $ (Q^{-\fsqrt}_{T-t} Q^{\fsqrt}_{\infty})$ and $ (Q^{-\fsqrt}_{\infty} S_{T+S-t})$ are strongly continuous in $t$, it follows from the uniform boundedness principle that 
\begin{align*}
    \sup_{t \in [0,S]} \|\Gamma_{T-t}\|_{HS} \leq \sup_{t \in [0,S]} \|(Q^{-\fsqrt}_{T-t} Q^{\fsqrt}_{\infty}) (Q^{-\fsqrt}_{\infty} S_{T+S-t})\| \|S_{T-S}\|_{HS} < \infty.
\end{align*}
\end{proof}

\begin{lemma}
\label{applem: templabel}
    For any $S < T$, the random process
    \begin{align}
        \Gamma_{T-t}^* Q^{-\fsqrt}_{T-t} y = \sum_{j=1}^{\infty} q_{j, T-t}^{-\fsqrt} \langle y, e_j \rangle \Gamma_{T-t}^* e_j, \quad t \in [0,S],
    \end{align}
    is well-defined as a limit in $L^2(H, \nu; C([0,S];H))$.
    Moreover, there exists a measurable space $H_S$ with $\nu(H_S) = 1$ such that the limit exists pointwise for $\nu$-a.e. $y \in H_S.$
\end{lemma}
\begin{proof}
For any $n \in \N$ define the process  
\begin{align*}
    \Upsilon_n(t) = \sum_{j = 1}^n q_{j, T-t}^{-\fsqrt} \langle y, e_j \rangle \Gamma^*_{T-t} e_j
\end{align*}
where $(q_{j,T-t}, e_j)_j$ is the eigenbasis of $Q_{T-t}$. 
From the strong continuity of $(Q_t)_t$ and $(\Gamma_t)_t$ it follows that $\Upsilon_n \in C([0,S];H)$ for any $n \in \N$.
Furthermore, it holds that 
\begin{align*}
    \int_H  \| \Upsilon_n \|_{0}^2 \, \nu(\df y) &= 
    \int_H  \sup_{t \in [0,S]} \| \sum_{j = 1}^n q_{j, T-t}^{-\fsqrt} \langle y, e_j \rangle \Gamma^*_{T-t} e_j  \|^2 \, \nu(\df y) \\
    &=   \sup_{t \in [0,S]} \sum_{j = 1}^n q_{j,T-t}^{-1} \left( \int_H | \langle y, e_j \rangle|^2 \, \nu(\df y) \right) ~\| \Gamma^*_{T-t} e_j \|^2 \\
    &= \sup_{t \in [0,S]} \sum_{j = 1}^n q_{j,T-t}^{-1} \langle Q_{\infty} e_j, e_j \rangle ~\| \Gamma^*_{T-t} e_j \|^2 \\
    &= \sup_{t \in [0,S]} \sum_{j = 1}^n \langle Q^{-\fsqrt}_{T-t} Q^{\fsqrt}_{\infty} e_j,Q^{-\fsqrt}_{T-t} Q^{\fsqrt}_{\infty} e_j \rangle \langle \Gamma^*_{T-t} e_j, \Gamma^*_{T-t} e_j \rangle \\
    &\leq \sup_{t \in [0,S]} \left( \| Q^{-\fsqrt}_{T-t} Q^{\fsqrt}_{\infty}\| ~ \|\Gamma^*_{T-t} \|^2_{HS} \right) < \infty
\end{align*}
from which we conclude the convergence of $\Upsilon_n \to \Gamma_{T-t}^* Q^{-\fsqrt}_{T-t} y$ in $L^2(H, \nu; C([0,S];H)).$
Now, the second claim follows by an application of the Itô–Nisio theorem. 
\end{proof}

\begin{corollary}
There exists a measurable space $H_0 \subset H$ with $\nu(H_0) = 1$ such that 
\begin{align}
    [0,T) \to H, ~t \mapsto \Gamma_{T-t}^* Q^{-\fsqrt}_{T-t} y
\end{align}
is well-defined and continuous for any $y \in H_0.$
\end{corollary}
\begin{proof}
    Set $H_0 = \bigcap_n H_{T-1/n}$ where $H_{T-1/n}$ is the measurable space given by Lemma \ref{applem: templabel}.
\end{proof}

\end{appendices}

\end{document}